\documentclass[letter,11pt]{article}
\usepackage{amsmath} %equation*,pmatrix,cases
\usepackage{amssymb}  % \mathbb
\usepackage{amsthm}
\usepackage[dvips]{graphicx}
\usepackage{times}
\newcommand{\Order}{\mathrm{O}}

\newcommand{\defeq}{\stackrel{\mbox{\scriptsize{\normalfont\rmfamily def.}}}{=}}

\renewcommand{\Vec}[1]{\mbox{\boldmath $#1$}}
\newcommand{\tree}{G}
\newcommand{\frakd}{{\mathfrak{d}}}
\newtheorem{theorem}{Theorem}[section]
\newtheorem{lemma}[theorem]{Lemma}
\newtheorem{corollary}[theorem]{Corollary}
\newtheorem{proposition}[theorem]{Proposition}

\newtheorem{observation}[theorem]{Observation}

\title{
An Analysis of the Recurrence/Transience 
of\\Random Walks on Growing Trees and Hypercubes %by Coupling
}
\author{
 Shuma Kumamoto\footnote{
   Graduate School of Mathematical Science, Kyushu University
   } \and
  Shuji Kijima\footnote{
  Department of Data Science, Shiga University
   }\and
  Tomoyuki Shirai\footnote{
  Institute of Mathematics for Industry, Kyushu University
  }
}

\begin{document}
\maketitle
\begin{abstract}
 It is a celebrated fact that 
   a simple random walk on an \emph{infinite} $k$-ary tree for $k \geq 2$ 
   returns to the initial vertex at most finitely many times during infinitely many transitions; it is called \emph{transient}. 
 This work points out the fact that 
   a simple random walk on an infinitely \emph{growing} $k$-ary tree 
   can return to the initial vertex infinitely many times, it is called \emph{recurrent}, 
   depending on the growing speed of the tree. 
 Precisely, 
   this paper is concerned with 
   a simple specific model of a \emph{random walk on a growing graph} (\emph{RWoGG}), and 
   shows a phase transition between the recurrence and transience of the random walk 
     regarding the growing speed of the graph. 
 To prove the phase transition,
   we develop a coupling argument,  
    introducing the notion of \emph{less homesick as graph growing} (\emph{LHaGG}). 
 We also show some other examples, 
   including a random walk on $\{0,1\}^n$ with infinitely growing $n$, 
  of the phase transition between the recurrence and transience. 
 We remark that some graphs concerned in this paper have infinitely growing degrees.

%%%%%%%%%%
%\keywords{
\noindent
{\bf Keywords: }
Random walk, 
dynamic graph, 
recurrent, transient
%uniform bounded
%dynamic graph,
%changing graph,
%bipartite graph
%}
\end{abstract}

\section{Introduction}
 The recurrence or transience is 
    a classical and fundamental topic of random walks on {\em infinite} graphs, see e.g.,~\cite{Durrett}: 
 let $X_0,X_1,X_2,\ldots$ be a random walk (or a Markov chain)\footnote{
   This paper is concerned with discrete time and space processes. 
   We will be mainly concerned with time-{\em inhomogeneous} Markov chains, 
   but here you may assume a time-homogeneous chain, 
   i.e., the transition probability $\Pr[X_{t+1} = v \mid X_t=u]$ is independent of the time $t$, but depends on $u,v$.} 
   on an infinite state space $V$, e.g., $V=\mathbb{Z}$, 
  with $X_0 = v$ for $v \in V$. 
 For convenience, 
  let 
\begin{align*}
  R(t) = \Pr[X_t = v] \hspace{2em}(= \Pr[X_t = v \mid X_0 = v]) 
\end{align*}  
   denote the probability that a random walk returns to the initial state at time step $t$ ($t=1,2,\ldots$), 
 and then the initial point $v$ is {\em recurrent} by the random walk  if 
\begin{align}
 \sum_{t=1}^{\infty} R(t) = \infty
\label{def:recurrent0}
\end{align}
 holds, otherwise it is {\em transient}. 
 Intuitively, \eqref{def:recurrent0} means that 
   the random walk is ``expected'' to return to the initial state infinitely many times. 
 It is well known that 
   a simple random walk on $\mathbb{Z}^d$ is recurrent for $d=1,2$, 
  while it is transient for $d \geq 3$, cf.~\cite{Durrett}. 
 Another celebrated fact is that 
  a simple random walk on an infinite $k$-ary tree is transient~\cite{Lyons90,LyonsPeres}. 

%%%
 Analysis of random walks on \emph{dynamic graphs} 
     has been developed in several contexts. 
%%%%
 In probability theory, 
   random walks in random environments are a major topic, 
   where self-interacting random walks 
   including reinforced random walks and excited random walks 
   have been intensively investigated 
  as a relatively tractable non-Markovian process, see e.g., \cite{Davis90,BW03,DKL08,SZ09,SZ11,KZ13}. 
 The recurrence or transience of a random walk in a random environment 
   is a major topic there, 
   particularly random walks on growing subgraphs of $\mathbb{Z}^d$ 
   or on infinitely growing trees are the major targets \cite{DHS14C,DHS14J,Huang19,Amir20}. 
%%%%
 In distributed computing, 
   analysis of algorithms, including random walk, on dynamic graph  
   attracts increasing attention 
  due to the fact that real networks are often dynamic~\cite{Cooper11,KO11,APR16,MS18}. 
 Searching or covering networks, related to hitting or cover times of random walks, 
    are major topics there~\cite{CF03,AKL08,DR14,AKL18,LMS18,CSZ20,KSS21}.

%%%%%
 This work is concerned with the recurrence/transience of a random walk on a growing graph. 
 We show the fact that 
    a simple random walk on an infinitely growing complete $k$-ary tree 
   can be recurrent 
   \emph{depending on the growing speed} of the tree, 
 while 
   a simple random walk on an infinite $k$-ary tree is transient as we mentioned above. 
%%%%
 More precisely, 
   this paper follows the model of the random walk on growing graph (RWoGG) \cite{KSS21},  
    where the network gradually grows 
     such that the growing network keeps its shape $G(n)$ for $\mathfrak{d}(n)$ steps, and then 
   changes the shape to $G(n+1)$ by adding some vertices to $G(n)$  (see Section~\ref{sec:model} for detail). 
%%%
 Then, we show
  a phase transition between the recurrence and transience of 
   a random walk on a growing $k$-ary tree, 
     regarding the growing speed of the graph. 
%%%
 For a proof,
   we develop the notion of 
     {\em less-homesick as graph growing} (LHaGG), 
    which is a quite natural property of RWoGG, and 
    gives a simple proof by a {\em coupling} argument, 
   that is an elementary technique of random walks or Markov chains based on a comparison method.  
%%%%
 We also show some other examples of the phase transition, 
   including as a random walk on $\{0,1\}^n$ with infinitely growing $n$.

\subsection{Existing works and contribution of the paper}
 The recurrence/transience of a random walk on a dynamic graph 
    has been mainly developed in the context of random walks in random environment 
    including reinforced random walks and excited walks. 
%%%
 Here we briefly review some existing works 
   concerning the recurrence of a random walk on  $\mathbb{Z}^d$ and infinite (or infinitely growing) trees, 
  directly related to this paper. 

\subparagraph{Random walks on (asymptotically) $\mathbb{Z}^d$. }
 It is a celebrated fact that the initial point, say origin $\Vec{0}$, 
   in the infinite integer grid $\mathbb{Z}^d$ 
   is recurrent when $d=1$ and $2$ by a simple random walk, and 
   it is transient for $d \geq 3$, see e.g., \cite{Durrett}. 

 Dembo et al.~\cite{DHS14J} is concerned with 
   a random walk on an infinitely growing subgraph of $\mathbb{Z}^d$, and 
   gave a phase transition, that is roughly speaking 
    a random walk is recurrent if and only if 
    $\sum_{t=1}^{\infty} \pi_t(\Vec{0}) = \infty$ holds under a certain condition, 
   where $\pi_t$ denotes the stationary distribution of the transition matrix at time $t$. 
  Huang~\cite{Huang19} extended the argument of \cite{DHS14J} and 
   gave a similar or essentially the same phase transition for more general graphs. 
%%%%
 The proofs are based on the edge conductance and a central limit theorem, 
   on the assumptions that 
   every vertex of the dynamic graph has a degree at most \emph{constant} to time (or the size of the graph), and 
   the random walk is ``lazy'' such that it has at least a \emph{constant} probability of self-loops at every vertex.  
%%%
 Those arguments are sophisticated and enhanced 
    using the argument of evolving set and the heat kernel by recent works \cite{DHMP17,DHZ19}.

\subparagraph{Random walks on infinitely growing trees. }
 Lyons~\cite{Lyons90} studied sufficient conditions for a random walk being recurrent/transient, see also \cite{LyonsPeres}. 
 Roughly speaking, 
    the initial point, say the root $r$, is recurrent if and only if the random walk is enough {\em homesick}, 
   meaning that a random walk probabilistically tends to choose the direction to the root. 

%%%%%
 Amir et al.~\cite{Amir20} introduced a random walk in changing environment model, and 
   investigated the recurrence and transience of random walks in  the model. 
 They gave a conjecture about the conditions for the recurrence and transience regarding the limit of a graph sequence, and 
  proved it for trees. 
 Huang's work \cite{Huang19}, which we mentioned above, implies that a simple random walk starting from a vertex $v$ on growing $k$-ary tree is recurrent 
  if and only if $\sum \pi_t(v) = \infty$, that is similar to or essentially the same as a main result of this paper 
    under a certain condition.  
 We remark that a $k$-ary tree with height $n$ is not an (edge induced) subgraph of $\mathbb{Z}^d$ for a {\em constant} $d$. 

%%%%%
 There is a lot of work on the recurrence or transience of a random walk on a growing tree, 
   related to self-interacting random walks 
   including reinforced random walks and excited random walks, e.g.,~\cite{IFN19,FIORR21}. 
 They are non-Markovian processes, and in a bit different line from \cite{DHS14J,Amir20,Huang19} and this paper.

\subparagraph{Contribution of this work. }
 This paper is concerned with 
   a specific model of dynamic graphs with an increasing number of vertices, which we will describe in Section~\ref{sec:model}, and 
   gives a phase  transition by the growing speed regarding a random walk being recurrent/transient.  
 The phase transition is very similar to or essentially the same as \cite{DHS14J,Huang19}, 
   while this paper contains mainly three contributions. 
 One is the proof technique:  
   we employ a coupling argument 
   while the existing works are based on the conductance and a central limit theorem. 
 The coupling arguments is a classical and elementary comparison technique of random walks, 
   and we introduce the notion of LHaGG to use the comparison technique. 
 Since the coupling technique is relatively simple, 
   we can drop two assumptions in the existing works, 
     namely a random walk being {\em lazy} and a growing graph having {\em uniformly bounded degree}, 
     which are naturally required in the conductance argument to make the arguments simple. 
 This paper is mainly concerned with reversible random walks of {\em period 2}, 
    which contains simple random walks on undirected bipartite graphs; 
   this is the second contribution. 
 We also show an example of random walk on $\{0,1\}^n$ with increasing $n$, 
   where the (maximum) degree of the dynamic graph, that is $n$, infinitely grows; 
   this is the third contribution. 

 While the coupling technique is relatively easy, 
   it often selects the applicable target.  
 In fact, 
   the results by \cite{DHS14J,Huang19} 
  are widely applied to general setting as far as it satisfies appropriate assumptions, 
  while our result is limited to specific targets. 
 Such an argument about conductance and coupling seems  
   known as an implicit knowledge in the literature of mixing time analysis, cf.~\cite{AR01,Guruswami}. 
 However, 
   we emphasize that 
     the coupling technique 
     often gives an easy proof of an interesting phenomena, 
     as this paper shows.

\subsection{Organization}
 As a preliminary, we describe the model of random walk on growing graph (RWoGG) in Section~\ref{sec:pre}. 
 Section~\ref{sec:frame} introduces the notion of less homesickness as graph growing (LHaGG), and 
   presents some general theorems for sufficient conditions of a RWoGG being recurrent/transient. 
 Section~\ref{sec:karytree} shows 
   a phase transition between the recurrence and transience of a random walk on growing $k$-ary tree. 
 Section~\ref{sec:cube} shows a phase transition for a random walk on $\{0,1\}^n$ with increasing $n$. 
 Some other examples, such as growing region of $\mathbb{Z}^d$ and another growing tree are found in appendix.

\section{Preliminaries}\label{sec:pre}
\subsection{Model}\label{sec:model}
 A growing graph is a sequence of (static) 
  graphs $\Vec{\cal G} = {\cal G}_0, {\cal G}_1, {\cal G}_2, \ldots$
 where ${\cal G}_t=({\cal V}_t,{\cal E}_t)$ for $t=0,1,2,\ldots$ denotes 
  a graph\footnote{
    Every static graph is simple and undirected in this paper, 
     for simplicity of the arguments. } 
   with a finite vertex set ${\cal V}_t$ and an edge set ${\cal  E}_t \subseteq {{\cal V}_t \choose 2}$. 
 For simplicity, 
 this paper assumes\footnote{
  Thus, the current position does not disappear in the next step. 
  }   
  ${\cal V}_t \subseteq {\cal V}_{t+1}$ and ${\cal E}_t \subseteq {\cal E}_{t+1}$. 
 In this paper, 
   we assume $|{\cal V}_{\infty}|=\infty$, otherwise the subject (recurrence) is trivial. 
 A random walk on a growing graph is a Markovian series $X_t \in {\cal V}_t$ ($t=0,1,2,\ldots$). 

 In particular, this paper is concerned with a specific model, described as follows, cf.~\cite{KSS21}.  
 A {\em random walk on a growing graph} ({\em RWoGG}), in this paper, is 
   formally characterized by a 3-tuple of functions ${\cal D} = \left(\mathfrak{d},G,P\right)$.
 The function $\mathfrak{d}\colon \mathbb{Z}_{>0} \to \mathbb{Z}_{ \geq 0}$ denotes the duration.
 For convenience, let $T_n=\sum_{i=1}^n \frakd(i)$ for $n=1,2,\ldots$\footnote{
   We do not  exclude $T_{n-1}=T_n$; if $\frakd(n)=0$ then  $T_{n-1}=T_n$. 
  } and $T_0=0$. 
 We call the time interval
 $[T_{n-1},T_n]$ {\em phase} $n$ for $n=1,2,\ldots$; 
   thus $T_{n-1}=\sum_{i=1}^{n-1} \frakd(i)$ is the beginning of the $n$-th phase, but 
     we also say that $T_{n-1}$ is the end of the $(n-1)$-st phase, for convenience. 
%%%
 The function $G\colon \mathbb{Z}_{>0} \to \mathfrak{G}$ 
   represents the graph $G(n) = (V(n),E(n))$ for the phase $n$, 
  where $\mathfrak{G}$ denotes the set of all (static) graphs, 
    i.e., our growing graph $\Vec{\cal G}$ satisfies ${\cal G}_t = G(n)$ for $t \in [T_{n-1},T_n)$. 
%%%
 Similarly, the function 
  $P\colon \mathbb{Z}_{>0} \to \mathfrak{M}$ is a function that represents the ``transition probability'' 
   of a random walk on graph $G(n)$ 
   where $\mathfrak{M}$ denotes the set of all stochastic matrices.

 A RWoGG $X_t$ ($t=0,1,2,\ldots$) characterized by ${\cal D} = \left(\mathfrak{d},G,P\right)$ is 
   temporally a time-homogeneous finite Markov chain 
   according to $P(n)$ with the state space $V(n)$ during the time interval  $[T_{n-1},T_{n}]$; 
   precisely, a transition from $X_t$ to $X_{t+1}$ follows  $P(n)$ for any  $t \in [T_{n-1},T_{n})$. 
 We specially remark for $t=T_n$ that $X_t \in V(n) \subseteq V(n+1)$, 
  meaning that $X_t$ is a state of $V(n+1)$ but actually $X_t$ must be in $V(n)$ by the definition of the transition. 
  Suppose $X_0=v$ for $v \in V(1)$. 
We define the return probability at $v$ by 
\begin{align}
 R(t) = \Pr[X_t = v]\ \left(= \Pr[X_t = v \mid X_0=v]\right)
\end{align}
at each time $t = 0,1,2,\ldots$. 
We say $v$ is \emph{recurrent} by RWoGG ${\cal D}=(\frakd,G,P)$ if 
\begin{align}
  \sum_{t=1}^{\infty}R(t) = \infty
\label{def:recurrent}
\end{align}
  holds, 
 otherwise, i.e., $\sum_{t=0}^{\infty}R(t)$ is finite, 
  $v$ is {\em transient} by ${\cal D}$. 

\subsection{Terminology on time-homogeneous Markov chains}\label{sec:term}
 We here briefly introduce some terminology 
    for random walks on static graphs, or time-homogeneous Markov chains, according to~\cite{LevinPeres}. 
\subsubsection{Ergodic random walks}\label{sec:term1}
%%%%%%%%%%%%%
 Suppose that $X_0,X_1,X_2,\ldots$  is a random walk on a static graph $G=(V,E)$ 
   characterized by a time-homogeneous transition matrix $P = (P(u,v))\in \mathbb{R}_{\geq 0}^{V\times V}$
    where $P(u,v)=\Pr[X_{t+1}=v \mid X_t=u]$. 
%%%%%%%%%
 A random walk 
  is \emph{reversible} if 
   there exists a positive function $\mu \colon V \to \mathbb{R}_{>0}$ such that 
   $\mu(u)P(u,v)=\mu(v)P(v,u)$ hold for all $u,v\in V$. 
 A transition matrix $P$ 
   is \emph{irreducible} 
     if $\forall u,v \in V$, $\exists t>0$, $(P^t)(u,v) > 0$. 
 The period of $P$ is given by 
   ${\rm period}(P) = \min_{v \in V} {\rm gcd}\{t>0 : (P^t)(v,v)>0\} $. 
 It is well known that ${\rm gcd}\{t>0 : (P^t)(v,v)>0\}$ is common for any $v \in V$ if $P$ is irreducible. 

 If  ${\rm period}(P)=1$ then $P$ is said to be {\em aperiodic}. 
 A transition matrix $P$ is {\em ergodic} if it is irreducible and aperiodic. 
 We say a random walk is ($\gamma$-){\em lazy} 
   if $P(v,v) \geq \gamma$ holds for any $v \in V$ for a constant $\gamma$ ($0 < \gamma < 1$). 
 A lazy random walk is clearly aperiodic. 
 A probability distribution $\pi$ over $V$ is a \emph{stationary distribution} 
  if it satisfies $\pi P = \pi$.  
 It is well known that an ergodic $P$ has a unique stationary distribution~\cite{LevinPeres}. 
%%%%%%%%
 The {\em mixing time} of $P$ is given by
\begin{align} 
  \tau(\epsilon) \defeq \min\left\{t \ \middle|\ t \in \mathbb{Z}_{>0},\ \frac{1}{2}\max_{u \in V}\sum_{v\in V}\left|P^t(u,v)-\pi(v)\right|\leq \epsilon\right\}
\end{align}  
for $\epsilon \in(0,1)$.

\subsubsection{Random walk with period 2}\label{sec:term2}
  A {\em simple} random walk (or ``busy'' simple random walk) 
    on an undirected graph $G=(V,E)$ is given by $P(u,v) = 1/\deg(u)$ for $\{u,v\} \in E$
    where $\deg(u)$ denotes the degree of $u \in V$ on $G$. 
 This paper is mainly concerned with bipartite graphs, 
    such as trees, integer grids, and $0$-$1$ hypercubes, and then 
  the most targeted random walks are irreducible and reversible, but {\em not aperiodic}. 
\begin{observation}
  If $P$ is reversible then its period is at most 2. 
\end{observation}
Suppose $P$ is irreducible and reversible, and it has period 2. 
Then,  the underlying graph is a connected bipartite $(U,\overline{U};E)$, 
  where 
   $U = \{ u \in V \mid \exists t',\ P^{2t'}(v,u) \neq 0 \}$ for any $v \in U$, 
   $\overline{U} = \{ u \mid \forall t,\ P^{2t}(v,u) = 0\}$, i.e.,  $\overline{U} = V \setminus \overline{U}$, and  
   $E=\{ \{u,v\} \in V^2 \mid P(u,v) > 0\}$. 
 Notice that $E$ does not contain any self-loop, otherwise, $P$ is aperiodic. 

  Here, we introduce some unfamiliar terminology for periodic Markov chains. 
We say $\mathring{x} \in \mathbb{R}^{V}_{\geq 0}$ is 
 {\em even-time distribution} if it satisfies 
   $\sum_{v \in V} \mathring{x}(v) = 1$ and 
   $\mathring{x}(u)=0$ for any $u \in \overline{U}$. 
We say $\mathring{\pi} \in \mathbb{R}^V_{\geq0}$ is {\em even-time stationary distribution} 
 if it is an even-time distribution and satisfies
  $\mathring{\pi} P^2 = \mathring{\pi}$. 
\begin{proposition}[limit distribution]
Suppose $P$ is irreducible and reversible, and it has period 2. 
Then, $P$ has a unique even-time stationary distribution $\mathring{\pi}$, and 
 $\lim_{t \to \infty} \mathring{x} P^{2t} = \mathring{\pi}$ for any even-time distribution $\mathring{x}$.  
\end{proposition}

We define the {\em even mixing-time} of $P$ by
\begin{align}
\mathring{\tau}(\epsilon) = \min\left\{ 2t'  \ \middle|\ t' \in \mathbb{Z}_{>0},\ \frac{1}{2} \max_{u \in U}  \sum_{v\in U} \left| P^{2t'}(u,v) - \mathring{\pi}(v)\right| \leq \epsilon \right\}
\end{align}
 for $\epsilon \in (0,1)$. 
 We remark that the even mixing-time of $P$ is equal to the twice of the mixing time of $P^2[U]$, 
  where $P^2[U]$ denotes the submatrix of $P$ induced by $U$.   
 Thus, we can use some standard arguments, e.g., coupling technique, about the even mixing-time of $P$. 
 Finally, we remark on a proposition, that plays a key role in our analysis.

\begin{proposition}[Proposition 10.25 in \cite{LevinPeres}]\label{prop:reversible-monotone}
If $P$ is reversible then $ \mathring{\pi}(v) \leq P^{2t+2}(v,v) \leq P^{2t}(v,v)$ for any $t = 0,1,2,\ldots$. 
\end{proposition}

\section{Analytical Framework: LHaGG}\label{sec:frame}
 This section introduces the notion of less-homesickness as graph growing (LHaGG), and  
   presents general theorems (Lemmas~\ref{lem:rec} and \ref{lem:trans}) 
     describing some sufficient conditions of a RWoGG being recurrent or transient. 
 See the following sections for specific RWoGGs, 
   namely, 
    RW on growing $k$-ary tree in Section~\ref{sec:karytree}, 
    RW on $\{0,1\}^n$ hypercube skeleton with increasing $n$ in Section~\ref{sec:cube}, etc. 

\subsection{Less-homesick as graph growing}\label{sec:lhagg}
 Let ${\cal D} = (f,G,P)$ and  ${\cal D}' = (f',G',P')$ be RWoGG, and let 
  $R(t)$ and $R'(t)$ respectively denote their return probabilities to respective initial vertices at time $t = 1,2,\ldots$. 
 We say ${\cal D}$ is {\em less-homesick} than ${\cal D}' = (f',G',P')$ at time $t$ if 
 $R(t) \leq R'(t)$
 holds. 

%%%
 In particular, 
  this paper is mainly concerned with the less-homesick relationship between 
  ${\cal D} = (f,G,P)$ and ${\cal D}' = (g,G,P)$
  with the same $P$, $G$ and the initial vertex $v$. 
 We say 
  ${\cal D}$ is {\em less-homesick as graph growing} ({\em LHaGG})\footnote{
   Strictly speaking, 
    LHaGG should be a property of the sequence of transition matrices  
    $P(1),P(2),P(3),\ldots$. 
   For the convenience of the notation, we say ${\cal D} = (f,G,P)$ is LHaGG, in this paper. 
  } 
  if ${\cal D} = (f,G,P)$ is less-homesick than for any ${\cal D}' = (g,G,P)$ 
 satisfying that 
\begin{align}
   \sum_{i=1}^n f(i) \leq \sum_{i=1}^n g(i)
\label{eq:grow-faster}
\end{align}
   for any $n\in \mathbb{Z}_{>0}$. 
 The condition \eqref{eq:grow-faster} intuitively implies 
  that the graph in  ${\cal D}$ grows faster than ${\cal D}'$. 
 For instance, we will prove that the simple random walk on growing $k$-regular tree is  LHaGG, in Section~\ref{sec:karytree}.

\begin{lemma}
\label{lemma:comp-static}
 Suppose RWoGG ${\cal D} = (f,G,P)$ is LHaGG. 
 Let $X_t$ ($t=0,1,2,\ldots$) be a RWoGG according to ${\cal D}$ with $X_0 = v \in V(1)$.
 Let $Y_t$  ($t=0,1,2,\ldots$) be a random walk on (a static graph) $G(n)$ according to $P(n)$
 with $Y_0=v$, where 
  $G$, $P$ and $v$ are common with  ${\cal D}$. 
 Then, $Y_t$ is less-homesick than $X_t$ at any time $t \in [T_n, T_{n+1}]$, i.e., 
$ R(t) \geq R'(t)$
holds for $t \in [T_n, T_{n+1}]$,  
where $R(t)=\Pr[X_t = v]$ and   $R'(t)=\Pr[Y_t = v]$. 
\end{lemma}
\begin{proof}
Let 
\begin{align*}
g(i) = \begin{cases}
0 & (i < n), \\
\sum_{j=1}^n f(j) & (i = n), \\
f(i) & (i > n).
\end{cases}
\end{align*}
 Then, the static random walk $Y_t$ on $G(n)$ also follows ${\cal D}' = (g,G,P)$ for $t \leq T_{n+1}$. 
 Clearly, $\sum_{i=1}^n f(i) \geq \sum_{i=1}^n g(i)$ for any $n$. 
 Since ${\cal D}$ is LHaGG by the hypothesis, 
$ R(t) \geq R'(t)$. 
\end{proof}

We remark that if all $P_n$ takes period 2 then $R(t) = R'(t) = 0$ for any odd $t$.

\subsection{Recurrent}\label{sec:rec}
We prove the following lemma, presenting a sufficient condition for a RWoGG to be recurrent. 
\begin{lemma}\label{lem:rec}
Suppose that 
 RWoGG ${\cal D} = (\mathfrak{d},G,P)$ is LHaGG, and 
 that every $P(n)=P_n$ ($n = 1, 2, \ldots$) is irreducible, reversible and ${\rm period}(P_n)=2$. 
 Let $p(n) = \mathring{\pi}_n(v)$ where $\mathring{\pi}_n$ denote the even-time stationary distribution of $P_n$. 
If $\mathfrak{d}$ satisfies 
\begin{align}
\sum_{n=1}^{\infty}(\mathfrak{d}(n) -1) p(n) = \infty  \label{eq:rec}
\end{align}
then $v$ is recurrent by ${\cal D}$.
\end{lemma}

\begin{proof}
 Let $f(n) = 2\lfloor \frac{\mathfrak{d}(n)}{2} \rfloor$, 
 i.e., $f(n)=\mathfrak{d}(n)$ if $\mathfrak{d}(n)$ is even, 
 otherwise $f(n)=\mathfrak{d}(n)-1$. 
 For convenience, let 
 $T'_n = \sum_{k=1}^n f(k)$
 for $n=1,2,\ldots$, and let $T'_0=0$. 
%%% 
 Let $X_t$ (resp. $X'_t$) for $t=0,1,2,\ldots$ be a RWoGG according to ${\cal D}=(\mathfrak{d},G,P)$ (resp. ${\cal D}'=(f,G,P)$), and 
  let $R(t)$  (resp. $R'$) denote the return probability of $X_t$ (resp.\ $X'_t$). 
 The hypothesis LHaGG implies $R(t) \geq R'(t)$. 
 Let $Y^n_t$ ($t=0,1,\ldots,T'_n$) be a time-homogeneous random walk according to $P(n)$, and 
 let $R''_n(t)$ ($t=1,\ldots,T'_n$) denote the return probability of $Y^n_t$. 
 The hypothesis LHaGG and Lemma~\ref{lemma:comp-static} implies 
\begin{align}
R'(t) \geq R''_n(t)
\label{eq:20231229a}
\end{align}
for $t  \in (T'_{n-1},T'_n]$. 
Then, we can see  
\begin{align}
\sum_{t=1}^{\infty} R(t) 
 &\geq \sum_{t=1}^{\infty} R'(t) && (\mbox{by LHaGG}) \nonumber \\
 &= \sum_{n=1}^{\infty}\sum_{t=T'_{n-1}+1}^{T'_n} R' (t) \nonumber \\
 &\geq \sum_{n=1}^{\infty}\sum_{t=T'_{n-1}+1}^{T'_n}R''_n(t) && (\mbox{by \eqref{eq:20231229a}}) \nonumber \\
 & = \sum_{n=1}^{\infty}\sum_{i=1}^{f(n)} R''_n(T'_{n-1} + i) &&(\mbox{recall $T'_n=T'_{n-1} +f(n)$}) \nonumber\\
 &= \sum_{n=1}^{\infty}\sum_{i'=1}^{\frac{f(n)}{2}} R''_n(T'_{n-1} + 2i') &&(\mbox{notice that $R''_n(T'_{n-1} + 2i'-1)=0$}) \nonumber \\
 &\geq \sum_{n=1}^{\infty}\sum_{i'=1}^{\frac{f(n)}{2}} p(n) && (\mbox{by Proposition~\ref{prop:reversible-monotone}})\nonumber \\
 &=\frac{1}{2}\sum_{n=1}^{\infty} f(n) p(n) \nonumber \\
 &\geq \frac{1}{2}\sum_{n=1}^{\infty} (\mathfrak{d}(n) -1) p(n) \label{eq:20240117a} 
\end{align}
hold. 
If \eqref{eq:rec} holds then \eqref{eq:20240117a} is $\infty$,
meaning that $v$ is recurrent by $\mathcal{D}$.
\end{proof}

The following lemma for aperiodic random walk is proved similarly. 
\begin{lemma}\label{lem:rec-ergodic}
Suppose that 
 RWoGG ${\cal D} = (\mathfrak{d},G,P)$ is LHaGG, and 
 that every $P(n)=P_n$ ($n = 1, 2, \ldots$) is irreducible, reversible and aperiodic, i.e., ergodic.  
 Let $p(n) = \pi_n(v)$ where $\pi_n$ denote the stationary distribution of $P_n$. 
 If $\mathfrak{d}$ satisfies 
\begin{align}
\sum_{n=1}^{\infty} \mathfrak{d}(n) p(n) = \infty  \label{eq:rec-aperiodic}
\end{align}
then $v$ is recurrent by ${\cal D}$.
\end{lemma}

\subsection{Transient}\label{sec:trans}
 This section establishes the following lemma, which suggests Lemma~\ref{lem:rec} is nearly optimal. 
 In fact, we will provide an example of a random walk on a growing $k$-ary tree in Section~\ref{sec:karytree}, 
  that shows  a tight example of Lemma~\ref{lem:rec}. 
\begin{lemma}\label{lem:trans}
 Suppose that a RWoGG ${\cal D} = (\mathfrak{d},G,P)$ is LHaGG, and 
  that every $P(n) =P_n$ ($n=1,2,\ldots$) is irreducible and reversible with ${\rm period}(P_n) = 2$. 
 Let $p(n) = \mathring{\pi}_n(v)$ where $\mathring{\pi}_n$ denote the even-time stationary distribution of $P_n$. 
 Let  $\mathring{\tau}_n(\epsilon)$ denote the even mixing-time of $P(n)$, and 
 let 
\begin{align*}
 \mathring{\mathfrak{t}}(n) = \mathring{\tau}_{n}(p(n))
\end{align*} 
for $n=2,3,\ldots$. 
If 
\begin{align}
\max\left\{\mathfrak{d}(1), \mathring{\mathfrak{t}}(1) \right\} + 
\sum_{n=2}^{\infty} 
 \max\left\{\mathfrak{d}(n), \mathring{\mathfrak{t}}(n) \right\} p(n-1) < \infty 
\label{eq:trans}
\end{align}
 holds then $v$ is transient by ${\cal D}$. 
\end{lemma}
\begin{proof}
Let 
\begin{align*}
f(n) = \max\left\{ \mathfrak{d}(n), \mathring{\mathfrak{t}}(n) \right\}
\end{align*}
for $n=1,2,3,\ldots$. 
Let 
$R(t)$ and $R'(t)$ respectively denote the return probabilities of 
${\cal D} =(\frakd,G,P)$ and ${\cal D}'=(f,G,P)$. 
Clearly, $f(n) \geq \mathfrak{d}(n)$ for any $n$, 
LHaGG implies 
\begin{align}
R(t) \leq R'(t) 
\label{eq:20230214b}
\end{align}
for any $t = 0,1,2,\ldots$. 
For convenience, let 
\begin{align}
T'_n = \sum_{k=1}^n f(k)
\end{align} 
for $n=1,2,\ldots$.

We carry a tricky argument in the following:
 roughly speaking we compare ${\cal D}'$ with $P_{n-1}$ in the $n$-th round, 
 i.e., $[T_{n-1},T_n]$, for $n=2,3,\ldots$. 
Let 
\begin{align*}
 g_{n-1}(k)=\begin{cases} f (k) & (k \leq n-2) \\ \infty & (k=n-1)\end{cases}
\end{align*}
 for $n=2,3,\ldots$. 
 Let $Z^{(n-1)}_t$ ($t=0,1,2,\ldots$) denote a RWoGG $(g_{n-1},G,P)$, where $Z^{(n-1)}_0=v$. 
 Let $R''_{n-1}(t)$ denote the return probability of $Z_t^{(n-1)}$, 
 Clearly, $\sum_{i=1}^j f(i) \leq \sum_{i=1}^j g_{n-1}(i) $ holds for any $j$, hence the LHaGG assumption implies 
 \begin{align}
  R'(t) \leq R''_{n-1}(t)
\label{eq:20231223a}
 \end{align}
   for any $t=0,1,2\ldots$ for any $n=2,3,\ldots$. 

 Notice that 
 $Z^{(n-1)}_t$ for $t \in [T_{n-2},T_n]$ is nothing but 
   a time-homogeneous random walk according to $P_{n-1}$ with the ``initial state'' $Z_{T_{n-2}}=v$
   for $n=2,3,\ldots$. 
Since 
\begin{align*}
  T'_{n-1} = T'_{n-2} + f(n-1) \geq T'_{n-2} + \mathring{\mathfrak{t}}(n-1) = T'_{n-2} + \mathring{\tau}_{n-1}(p(n-1))
\end{align*}
 $Z^{(n-1)}_t$ mixes well for $t > T'_{n-1}$,  
 meaning that $|\Pr[Z^{(n-1)}_t=v] - \mathring{\pi}_{n-1}(v)| \leq p(n-1)$ 
  for any even $t \in (T'_{n-1},T_n]$. 
This implies 
\begin{align}
R''_{n-1}(t) &= \Pr[Z^{(n-1)}_t=v] 
 \leq \mathring{\pi}_{n-1}(v) + p(n-1) 
  = 2p(n-1)
\label{eq:20231223b}
\end{align}
holds\footnote{
    We remark this argument requires only {\em point-wise additive error} bound, 
    instead of total variation. Clearly, point-wise additive error is upper bounded by total variation.  
   We here use the mixing time for total variation just because it has been better analyzed than the other. 
   \label{foot:additive-error}} 
 for $t \in (T'_{n-1},T'_n]$, where we remark that $R''_{n-1}(t) = 0$ for any odd $t$. 
%%%
 Then, 
 \begin{align*}
 \sum_{t=1}^{\infty}R(t) 
 &\leq \sum_{t=1}^{\infty}R'(t) && (\mbox{by \eqref{eq:20230214b}})\nonumber \\
 & = \sum_{n=1}^{\infty} \sum_{t=T'_{n-1}+1}^{T'_n}R'(t) \nonumber \\
 & \leq f(1) + \sum_{n=2}^{\infty} \sum_{t=T'_{n-1}+1}^{T'_n}R'(t) \nonumber \\
 & \leq f(1) + \sum_{n=2}^{\infty} \sum_{t=T'_{n-1}+1}^{T'_n}R''_{n-1}(t) 
 && (\mbox{by \eqref{eq:20231223a}}) \nonumber \\
 & \leq f(1) + \sum_{n=2}^{\infty}  \sum_{t=T'_{n-1}+1}^{T'_n} 2p(n-1)
 && (\mbox{by \eqref{eq:20231223b}}) \nonumber \\
&=  f(1) + 2\sum_{n=2}^{\infty} f(n)p(n-1) 
\end{align*}
holds. 
 Now it is easy to see that \eqref{eq:trans} implies $\sum_{t=1}^{\infty}R(t)<\infty $, meaning that $v$ is transient by $\mathcal{D}$. 
\end{proof}

The following lemma for aperiodic random walk is proved similarly. 
\begin{lemma}\label{lem:trans-ergodic}
 Suppose that a RWoGG ${\cal D} = (\mathfrak{d},G,P)$ is LHaGG, and 
  that every $P(n) =P_n$ ($n=1,2,\ldots$) is irreducible, reversible and aperiodic, i.e., ergodic. 
 Let $p(n) = \pi_n(v)$ where $\pi_n$ denote the stationary distribution of $P_n$. 
 Let  $\tau_n(\epsilon)$ denote the mixing time of $P(n)$, and 
 let 
\begin{align*}
 \mathfrak{t}(n) = \tau_{n}(p(n))
\end{align*} 
for $n=2,3,\ldots$. 
If 
\begin{align*}
\max\left\{\mathfrak{d}(1), \mathfrak{t}(1) \right\} + 
\sum_{n=2}^{\infty} 
 \max\left\{\mathfrak{d}(n), \mathfrak{t}(n) \right\} p(n-1) < \infty 
\end{align*}
 holds then $v$ is transient by ${\cal D}$. 
\end{lemma}

\paragraph{Remark.}
As we remarked at footnote \ref{foot:additive-error}, 
Lemma~\ref{lem:trans-ergodic} (similarly Lemma~\ref{lem:trans}, but here we omit) 
  is slightly improved using the mixing time by {\em point-wise additive error} as follows, but we do not use the fact in this paper. 

\begin{proposition}\label{prop:trans2}
 Suppose that a RWoGG ${\cal D} = (\mathfrak{d},G,P)$ is LHaGG, and 
  that every $P(n) =P_n$ ($n=1,2,\ldots$) is irreducible, reversible and aperiodic, i.e., ergodic. 
 Let $p(n) = \pi_n(v)$ where $\pi_n$ denote the stationary distribution of $P_n$. 
 Let
\begin{align*}
 \mathfrak{t}'(n) = \max_{u \in V} \left\{ t \in \mathbb{Z}_{>0} \mid \forall v \in V,\,  |P^t(u,v)-\pi(v) | \leq p(n) \right\}
\end{align*} 
for $n=2,3,\ldots$, 
If 
\begin{align*}
\sum_{n=1}^{\infty} 
 (\max\left\{\mathfrak{d}(n), \mathfrak{t}'(n) \right\} +1) p(n) < \infty 
\end{align*}
 holds then $v$ is transient by ${\cal D}$. 
\end{proposition}

\section{Random Walk on a Growing Complete $k$-ary Tree}\label{sec:karytree}
 Lyons gave sufficient conditions 
   that a random walk on an infinite tree gets recurrent or transient at the root (initial point), cf.~\cite{Lyons90,LyonsPeres}, 
  as a consequence, it is a celebrated fact that a simple random walk on an infinite $k$-ary tree is transient. 
 This section shows 
    that a simple random walk on a {\em moderately} growing complete $k$-ary tree 
   is recurrent at the root. 

\subsection{Result summary}
 Let $k$ be an integer greater than one, and 
  let $\tree_n = (V_n,E_n)$ denote a {\em complete $k$-ary tree} with height $n$ for $n=1,2,\ldots$, 
  i.e., $|V_n| = \sum_{i=0}^{n} k^i = \frac{k^{n+1} - 1}{k-1}$, 
     every internal node (including the root) has exactly $k$ children, and 
     every leaf places the same height $n$. 
 Let $r \in V_n$ denote the root, that is the unique vertex of height 0.  
 For convenience, let $h(v)$ denote the height of vertex $v \in V_n$, 
   i.e., $h(r) =0$, and $h(v)=n$ if and only if $v$ is a leaf of $\tree_n$.  
 Let 
\begin{align} 
 U_n = \{ v \in V_n \mid h(v) \equiv 0 \pmod{2}\}
\end{align}
  denote the vertices of even heights, and 
  thus $\overline{U}_n = V_n \setminus U_n$ is the vertices of odd heights. 
 Clearly, $\tree_n=(U_n,\overline{U}_n;E_n)$ is a bipartite graph.  
 See \cite{Cormen} for a standard terminology about a complete $k$-ary tree, 
  e.g., parent, child, root, internal node, leaf, height. 

%%%%
 Next, we define a transition probability of a random walk over $\tree_n$ according to \cite{Lyons90,LyonsPeres}.
 Let $\lambda$ be a fixed positive real\footnote{
   For simplicity of notation, Lyons~\cite{Lyons90} and Lyons and Peres~\cite{LyonsPeres} assume $\lambda>1$, 
   but many arguments are naturally extended to $\lambda >0$ by modifications with some bothering notations. 
 }, and we define a transition probability on the $k$-ary tree $\tree_n$ with height $n$ by 
\begin{align} 
P_n(u,v) =
\begin{cases}
\frac{1}{k} &\text{if $u = r$ and $v$ is a child of $u$,} \\
\frac{1}{\lambda+k} & \text{if $u \neq r$ and $v$ is a child of $u$,} \\
\frac{\lambda}{\lambda+k} & \text{if $u$ is an internal node and $v$ is the parent of $u$,} \\
1 &\text{if $u$ is a leaf and $v$ is the parent of $u$,}\\
0 & \text{otherwise,}
\end{cases}
\label{eq:comp-tree0}
\end{align}
   for $u,v \in V_n$. 
Notice that \eqref{eq:comp-tree0} denotes a {\em simple random walk} over $T_n$ when $\lambda =1$. 
 We also remark that $\lambda$ and $k$ are constants to $n$. 
 As a consequence of \cite{Lyons90},  
  we know the following fact about a random walk on an infinite $k$-ary tree $T_{\infty}$. 
\begin{proposition}[\cite{Lyons90,LyonsPeres}]\label{thm:infty-tree} 
 If $\lambda \geq k$ (resp.\ $\lambda < k$) then the root $r$ is recurrent (resp.\ transient) by $P_{\infty}$. 
\end{proposition}

 Then, we are concerned with a RWaGG ${\cal D}_{\rm T} = (\mathfrak{d},G,P)$ starting from the root $r$ where $G(n)=\tree_n$ and $P(n) = P_n$. 
%%%%
 Our goal of the section is to establish the following theorem. 
\begin{theorem}\label{thm:karytree}
 Let $k \geq 2$ and $\lambda > 0$ be constants to $n$.  
 Then, the root $r$ is recurrent by $\mathcal{D}_{\rm T}$ if
\begin{align}
 \sum_{n=1}^{\infty}\mathfrak{d}(n) \left(\frac{\lambda}{k}\right)^n
 =\infty
\end{align}
holds, otherwise, transient.
\end{theorem}

 For instance,  Theorem~\ref{thm:karytree} implies the following  corollary, 
   about a simple random walk on an infinitely growing $k$-ary tree. 
\begin{corollary}\label{cor:simple-karytree}
 Let $\lambda =1$, i.e., every $P_n$ denotes a {\em simple} random walk on the complete $k$-ary tree $T_n$. 
  If $\mathfrak{d}(n) = \Omega(k^n/(n \log n))$ then $r$ is recurrent by $\mathcal{D}_{\rm T}$. 
  If $\mathfrak{d}(1) < \infty$ and $\mathfrak{d}(n) = \Order(k^n/(n (\log n)^{1+\epsilon}))$ for $n\geq 2$ with a constant $\epsilon > 0$    then $r$ is transient by $\mathcal{D}_{\rm T}$. 
\end{corollary}
\begin{proof}
Suppose $\mathfrak{d}(n) \geq c k^n/(n \log n)$ for some constant $c>0$. 
Then, 
 $\sum_{n=1}^{\infty} \mathfrak{d}(n) (\frac{1}{k})^n \geq \sum_{n=1}^{\infty} c \frac{k^n}{n \log n} (\frac{1}{k})^n = c \sum_{n=1}^{\infty} \frac{1}{n \log n} \geq c \int_2^\infty \frac{1}{n \log n} = c[\log \log n]_2^{\infty} = \infty$ 
, and Theorem~\ref{thm:karytree} implies that $r$ is recurrent.

Suppose $\mathfrak{d}(n) \leq c' k^n/(n (\log n)^{1+\epsilon})$for some constant $c'>0$.  
Then,  $\sum_{n=1}^{\infty} \mathfrak{d}(n) (\frac{1}{k})^n 
   \leq \mathfrak{d}(1)+ \sum_{n=2}^{\infty} c' \frac{k^n}{n (\log n)^{1+\epsilon}} (\frac{1}{k})^n 
   \leq \mathfrak{d}(1) + c' \frac{1}{2 (\log 2)^{1+\epsilon}} + c'\int_2^{\infty} \frac{1}{x (\log x)^{1+\epsilon}} {\rm d}x
   = \mathfrak{d}(1) + c' \frac{1}{2 (\log 2)^{1+\epsilon}} + c'k \left[-\frac{1}{\epsilon (\log x)^{\epsilon}}\right]_2^{\infty}
   < \infty$,
and Theorem~\ref{thm:karytree} implies that $r$ is transient.
\end{proof}

\subsection{Proof of Theorem~\ref{thm:karytree}}\label{apx:proof-karytree}
 We prove Theorem~\ref{thm:karytree}. As a preliminary step, we remark on the following two facts. 
\begin{lemma}\label{lem:pn-karytree}
(i) Every $P_n$ ($n=1,2,\ldots$) is reversible: precisely, let 
\begin{align}
 \phi(v) = \begin{cases}
 \frac{k}{\lambda+k} & \text{if $h(v) =0$ (i.e., $v=r$),} \\
 \lambda^{-h(v)} & \text{if $0 < h(v) <n$,} \\
 \frac{\lambda}{\lambda+k}\lambda^{-n} & \text{if $h(v) = n$ (i.e., $v$ is a leaf). }
 \end{cases}
\label{eq:reversible-karytree}
\end{align}
Then, the detailed balance equation 
\begin{align*}
 \phi(u)P_n(u,v) = \phi(v)P_n(v,u)
\end{align*}
 holds for any $u,v \in V_n$. 
(ii) Every $P_n$ is irreducible and ${\rm period}(P_n) = 2$. 
 Thus the even-time stationary distribution of $P_n$ is 
\begin{align}
\mathring{\pi}_n(v) 
= \frac{\phi(v)}{\sum_{u \in U_n} \phi(u)}
\label{eq:stationary-karytree}
\end{align}
for any $v \in U_n$. 
\end{lemma}

Let $p(n) = \mathring{\pi}_n(r) $, then 
\begin{align}
 p(n) 
 = \begin{cases} 
 \frac{\frac{k}{\lambda+k}}{\frac{k}{\lambda+k} + \sum_{i=1}^{\lfloor \frac{n}{2} \rfloor } \left(\frac{k}{\lambda}\right)^{2i}}
 & \text{if $n$ is odd,} \\
 \frac{\frac{k}{\lambda+k}}{\frac{k}{\lambda+k} + \sum_{i=1}^{\frac{n}{2} -1} \left(\frac{k}{\lambda}\right)^{2i}  
  + \frac{\lambda }{\lambda + k}\left(\frac{k}{\lambda}\right)^n} 
 & \text{if $n$ is even} \\
\end{cases}
\label{eq:pi-regular_tree}
\end{align}
 by \eqref{eq:reversible-karytree} and \eqref{eq:stationary-karytree}
 considering the fact $|\{ v \in V_n \mid h(v) =i \}| = k^i$ for $i=0,1,\ldots,n$. 

\begin{lemma}\label{lem:pn-karytree-bound}
If $\lambda < k$, then 
\begin{align}
\tfrac{k-\lambda}{k}  \left(\frac{\lambda}{k}\right)^{n+1} \leq p(n) \leq \left(\frac{\lambda}{k}\right)^{n-1}.
\label{eq:pn-karytree-bound}
\end{align}
\end{lemma}
\begin{proof}
Firstly, we prove the upper bound of \eqref{eq:pn-karytree-bound}. 
When $n$ is odd,  
\begin{align*} 
p(n) 
= \frac{\frac{k}{\lambda+k}}{\frac{k}{\lambda+k} + \sum_{i=1}^{\lfloor \frac{n}{2} \rfloor } \left(\frac{k}{\lambda}\right)^{2i}}
\leq  \frac{1}{\frac{k}{\lambda+k} + \sum_{i=1}^{\lfloor \frac{n}{2} \rfloor } \left(\frac{k}{\lambda}\right)^{2i}}
\leq  \frac{1}{\left(\frac{k}{\lambda}\right)^{2\lfloor \frac{n}{2} \rfloor}}
=  \left(\frac{\lambda}{k}\right)^{2\lfloor \frac{n}{2} \rfloor}
=  \left(\frac{\lambda}{k}\right)^{n-1}
\end{align*}
and we obtain the upper bound in the case. 
When $n$ is even, similarly,  
\begin{align*}
p(n)
&=
 \frac{\frac{k}{\lambda+k}}{\frac{k}{\lambda+k} + \sum_{i=1}^{\frac{n}{2} -1} \left(\frac{k}{\lambda}\right)^{2i}  
  + \frac{\lambda }{\lambda + k}\left(\frac{k}{\lambda}\right)^n} 
=
 \frac{1}{1 + \frac{\lambda+k}{k}\sum_{i=1}^{\frac{n}{2} -1} \left(\frac{k}{\lambda}\right)^{2i}  
  + \left(\frac{k}{\lambda}\right)^n} 
\\&\leq
 \frac{1}{\left(\frac{k}{\lambda}\right)^n} 
\leq  \left(\frac{\lambda}{k}\right)^{n-1}
\end{align*}
and we obtain the upper bound. 
Then, we prove the lower bound of \eqref{eq:pn-karytree-bound}. 
When $n$ is odd,  
\begin{align*} 
p(n) 
&= \frac{\frac{k}{\lambda+k}}{\frac{k}{\lambda+k} + \sum_{i=1}^{\lfloor \frac{n}{2} \rfloor } \left(\frac{k}{\lambda}\right)^{2i}}
\geq \frac{\frac{k}{\lambda+k}}{1+ \sum_{i=1}^{\lfloor \frac{n}{2} \rfloor } \left(\frac{k}{\lambda}\right)^{2i}}
=
 \frac{\frac{k}{\lambda+k}}{
 \frac{\left(\left(\frac{k}{\lambda}\right)^2\right)^{\lfloor \frac{n}{2} \rfloor+1}-1}{\left(\frac{k}{\lambda}\right)^2-1}} 
=
 \frac{\frac{k}{\lambda+k}}{
 \frac{\left(\frac{k}{\lambda}\right)^{n+1}-1}{\left(\frac{k}{\lambda}\right)^2-1}} 
 \end{align*}
 and we obtain the lower bound in the case. 
 When $n$ is even, 
\begin{align*}
p(n)
=
 \frac{\frac{k}{\lambda+k}}{\frac{k}{\lambda+k} + \sum_{i=1}^{\frac{n}{2} -1} \left(\frac{k}{\lambda}\right)^{2i}  
  + \frac{\lambda }{\lambda + k}\left(\frac{k}{\lambda}\right)^n} 
\geq
 \frac{\frac{k}{\lambda+k}}{1 + \sum_{i=1}^{\frac{n}{2} -1} \left(\frac{k}{\lambda}\right)^{2i}  
  + \left(\frac{k}{\lambda}\right)^n} 
=
 \frac{\frac{k}{\lambda+k}}{
 \frac{\left(\frac{k}{\lambda}\right)^{n+2}-1}{\left(\frac{k}{\lambda}\right)^2-1}} 
\end{align*}
holds. In both cases, 
\begin{align*} 
p(n) 
&\geq
 \frac{\frac{k}{\lambda+k}}{
 \frac{\left(\frac{k}{\lambda}\right)^{n+2}-1}{\left(\frac{k}{\lambda}\right)^2-1}} 
= \tfrac{k}{\lambda+k} \left(\left(\tfrac{k}{\lambda}\right)^2-1 \right)
 \frac{1}{\left(\frac{k}{\lambda}\right)^{n+2}-1} 
\\&\geq \tfrac{k}{\lambda+k} \left(\left(\tfrac{k}{\lambda}\right)^2-1 \right)
 \frac{1}{\left(\frac{k}{\lambda}\right)^{n+2}} 
= \tfrac{k-\lambda}{k}  \left(\frac{\lambda}{k}\right)^{n+1}
 \end{align*}
holds and we obtain the lower bound. 
\end{proof}

The following lemma is a key of the proof of Theorem,~\ref{thm:karytree}. 
\begin{lemma}\label{lem:lhagg-karytree}
If $\lambda<k$ then 
$\mathcal{D}_{\rm T}$ is LHaGG. 
\end{lemma}
\begin{proof}
 Let $f$ and $g$ satisfy $\sum_{i=1}^n f(i) \leq \sum_{i=1}^n g(i)$ for any $n=1,2,\ldots$, and 
  let $X_t$ and $Y_t$ ($t=0,1,2,\ldots$) respectively follow $(f,G,P)$ and $(g,G,P)$, 
  i.e., the tree of $(f,G,P)$ grows faster than $(g,G,P)$. 
 Let $X_0 =Y_0 =r$, and 
  we prove $\Pr[X_t = r] \leq \Pr[Y_t = r]$ for any $t=1,2,\ldots$ (recall Section \ref{sec:lhagg} for LHaGG).

 We construct a coupling of $\Vec{X} = \{X_t\}_{t \geq 0}$ and  $\Vec{Y} = \{Y_t\}_{t \geq 0}$ 
  such that 
   $h(X_t) \geq h(Y_t)$ holds for any $t=1,2,\ldots$. % using the coupling technique.  
 The proof is an induction concerning $t$. 
 Clearly, $h(X_0) = h(Y_0)=0$. 
 Inductively assuming $h(X_t) \geq h(Y_t)$, we prove  $h(X_{t+1}) \geq h(Y_{t+1})$.  
 If $h(X_t) > h(Y_t)$ then $h(X_t) \geq h(Y_t)-2$ 
  since every $P_n$ is ${\rm period}(P_n)=2$ for $n=1,2,\ldots$. 
It is easy to see that 
$h(X_{t+1}) \geq h(X_t)-1 \geq h(Y_t)+1 \geq h(Y_{t+1})$, and we obtain $h(X_{t+1}) \geq h(Y_{t+1})$ in the case. 

Suppose $h(X_t)=h(Y_t)$. We consider four cases: 
(i) $X_t=Y_t=r$,  
(ii) both $X_t$ and $Y_t$ are internal nodes, 
(iii) both $X_t$ and $Y_t$ are leaves, i.e., both trees of $(f,G,P)$ and$(g,G,P)$ take the same height at time $t$, 
(iv) $X_t$ is not a leaf but $Y_t$ is a leaf, i.e., the tree of $(f,G,P)$ is higher than that of $(g,G,P)$ at time $t$. 
%%%
 In the case (i), 
\begin{align*}
\Pr[h(X_{t+1}) = h(X_t)+1] &= \Pr[h(Y_{t+1}) = h(Y_t)+1] = 1
\end{align*}
 hold, and hence we can couple them to satisfy $h(X_{t+1}) = h(Y_{t+1})$. % and hence $h(X') \geq h(Y')$. 
 In the case (ii), since both $X_t$ and $Y_t$ are internal nodes, 
\begin{align*}
\Pr[h(X_{t+1}) = h(X_t)-1] &= \Pr[h(Y_{t+1}) = h(Y_t)-1] = \frac{\lambda}{k+\lambda} 
\end{align*}
and
\begin{align*}
\Pr[h(X_{t+1}) = h(X_t)+1] &= \Pr[h(Y_{t+1}) = h(Y_t)+1] = \frac{k}{k+\lambda}
\end{align*}
 hold, and hence we can couple them to satisfy $h(X_{t+1}) = h(Y_{t+1})$. % and hence $h(X') \geq h(Y')$. 
%%%%
 In the case (iii), since both $X_t$ and $Y_t$ are leaves, 
\begin{align*}
\Pr[h(X_{t+1}) = h(X_t)-1] &= \Pr[h(Y_{t+1}) = h(Y_t)-1] = 1
\end{align*}
 holds, and hence we can couple them to satisfy $h(X_{t+1}) = h(Y_{t+1})$. % and hence $h(X') \geq h(Y')$. 
 In the case (iv), 
  since $X_t$ is not a leaf but $Y_t$ is a leaf, 
\begin{align*}
\Pr[h(X_{t+1}) = h(X_t)-1] &= \frac{\lambda}{k+\lambda} \leq \Pr[h(Y_{t+1}) = h(Y_t)-1] = 1 
\end{align*}
 holds, and hence we can couple them to satisfy $h(X_{t+1}) \geq h(Y_{t+1})$. 

 Now we obtain a coupling of $\Vec{X} = \{X_t\}_{t \geq 0}$ and  $\Vec{Y} = \{Y_t\}_{t \geq 0}$ 
  such that $h(X_t) \geq h(Y_t)$ hold for any $t=1,2,\ldots$, 
  which implies that $h(Y_t)=0$ as long as $h(X_t) = 0$. 
 This means that $\Pr[X_t =r] \leq \Pr[Y_t=r]$ for any $t=1,2,\ldots$.  
 We obtain the claim. 
\end{proof}

 By Lemma~\ref{lem:rec} with Lemma~\ref{lem:lhagg-karytree}, 
   we get a sufficient condition for recurrence in Theorem~\ref{thm:karytree}.  
 On the other hand, we cannot directly apply Lemma~\ref{lem:trans} 
   to the sufficient condition for transient in Theorem~\ref{thm:karytree}, 
   because the ``mixing time'' of $P_n$ is proportional to $k^n$, see e.g.,~\cite{LevinPeres}. 
 Then, we estimate $R(t)$ by another random walk. 

%%%
Let $Z_t =h(X_t)$, where $X_t$ is a random walk on a growing $k$-ary tree ${\cal D}_{\rm T} = (\mathfrak{d},G,P)$. 
Then $Z_t$ is a RWoGG ${\cal D}_{\rm L} = (\mathfrak{d},L,Q)$
where $L(n) = (\{0,1,\ldots,n\},\{\{i,i+1\} \mid i=0,1,2,\ldots,n-1\})$ is a path graph of length $n$,  
and the transition probability matrix $Q(n) = Q_n$ is given by 
\begin{align*}
\begin{cases}
Q_n(0,1)  =1, \\
Q_n(i,i+1) =\frac{k}{\lambda+k} &\mbox{for $i=1,2,\ldots,n-1$,}\\
Q_n(i,i-1) =\frac{\lambda}{\lambda+k} &\mbox{for $i=1,2,\ldots,n-1$,}\\
Q_n(n,n-1)  =1.
\end{cases}
\end{align*}
 The following Lemmas~\ref{lem:path-Rt=R't} and \ref{lem:path-lhagg} are easy to observe. 
\begin{lemma}\label{lem:path-Rt=R't}
 Let $X_t$ (resp. $Z_t$) follow $\mathcal{D}_{\rm T}=(f,G,P)$ (resp. $\mathcal{D}_{\rm L}=(f,L,Q)$). 
 Let $R(t) =\Pr[X_t = r]$ (resp.\ $R'(t) =\Pr[Z_t = r]$), and let  
 $\mathring{\pi}_n$ (resp.  $\mathring{\pi}'_n$) denote the even-time stationary distribution of $P_n$ (resp. $Q_n$). 
Then, $R(t) = R'(t)$ for any $t=1,2,\ldots$, as well as $\mathring{\pi}_n(r) = \mathring{\pi}'_n(r)$.  
\end{lemma}
\begin{lemma}\label{lem:path-lhagg}
If $\lambda<k$ then 
${\cal D}_{\rm L}$ is LHaGG. 
\end{lemma}
The following lemma about the mixing time of $Q_n$ is easily obtained 
  by a standard coupling argument for the mixing time, and we here omit the proof. 
\begin{lemma}\label{lem:path-mixing}
Let 
 $\mathring{\tau}'_n(\epsilon)$ denote the even mixing-time of $Q_n$ then 
 $\mathring{\tau}'_n(\epsilon) \leq n^2 \log \epsilon^{-1}$. 
\end{lemma}
%By a standard coupling argument for the mixing time, it is easy to see that 
%$\mathring{\tau}'_n(\epsilon) \leq n^2 \log \epsilon^{-1}$ 
% where $n^2$ is the hitting time of unbiased random walk on the path of length $n$. 
%\end{proof}
Then, we can prove the condition for ${\cal D}_{\rm L}$ being transient from Lemma~\ref{lem:trans}. 
\begin{lemma}\label{lem:path-trans}
 If $\lambda < k$ and $\sum \mathfrak{d}(n) (\frac{\lambda}{k})^n < \infty$ then $0$ is transient by ${\cal D}_{\rm L}$. 
\end{lemma}
\begin{proof}
Let $p(n) = \mathring{\pi}_n(r)$ and  $p'(n) = \mathring{\pi}'_n(r)$, 
 then $p(n) = p'(n)$ by Lemma~\ref{lem:path-Rt=R't}. 
By Lemma~\ref{lem:path-mixing}, 
$\mathring{\mathfrak{t}'}(n) = \mathring{\mathfrak{\tau}'}_n (p(n-1)) \leq n^2 \log (p(n-1)) \leq n^2 \log \!\left((\frac{\lambda}{k})^n\right) \leq c' n^3$, and hence 
$\sum_{n=1}^{\infty} \mathring{\mathfrak{t}'}(n) p(n-1) \leq  \sum_{n=1}^{\infty} n^3 c' (\frac{\lambda}{k})^{n-1} < \infty$. 
If $\sum_{n=1}^{\infty} \mathfrak{d}(n) (\frac{\lambda}{k})^n < \infty$, 
 then $\sum_{n=1}^{\infty} \max\{ \mathfrak{d}(n),  \mathring{\mathfrak{t}'}(n)\} p(n-1) \leq \sum_{n=1}^{\infty} (\mathfrak{d}(n) + \mathring{\mathfrak{t}'}(n)) (\frac{\lambda}{k})^{n-1} < \infty$, 
 which implies $\sum_{t=1}^{\infty} R'(t) < \infty$ by Lemma~\ref{lem:trans} with Lemma~\ref{lem:path-lhagg}. 
\end{proof}
Now, we are ready to prove  Theorem~\ref{thm:karytree}. 
\begin{proof}[Proof of Theorem~\ref{thm:karytree}]

First, we consider the (interesting) case  $\lambda < k$. 

(Recurrent)
Assuming $\sum_{n=1}^{\infty} \mathfrak{d}(n) (\frac{\lambda}{k})^n = \infty$, 
 we prove $\sum_{n=1}^{\infty} (\mathfrak{d}(n)-1) p(n) = \infty$. 
 Notice that  $\sum_{n=1}^{\infty}  p(n) \leq c \sum_{n=1}^{\infty}  (\frac{\lambda}{k})^n = \frac{1}{1-\frac{\lambda}{k}} < \infty$.  Let $C = \sum_{n=1}^{\infty}  p(n)$, then 
 $\sum_{n=1}^{\infty} (\mathfrak{d}(n)-1) p(n) = \sum_{n=1}^{\infty} \mathfrak{d}(n) p(n) -C \geq \sum_{n=1}^{\infty} \mathfrak{d}(n) c (\frac{\lambda}{k})^n -C$, which is $\infty$ 
 from the assumption. 
 Thus, $r$ is recurrent by Lemma~\ref{lem:rec}.
 
(Transient)
By Lemma~\ref{lem:path-Rt=R't}, $\sum_{t=1}^{\infty} R(t) = \sum_{t=1}^{\infty} R'(t)$. 
 If  $\sum_{n=1}^{\infty} \mathfrak{d}(n) (\frac{\lambda}{k})^n < \infty$ 
 then $\sum_{t=1}^{\infty} R'(t) < \infty$ by Lemma~\ref{lem:path-trans}, 
 meaning that $r$ is transient. 

 In the case of  $\lambda \geq k$, 
    it is always recurrent. 
 The proof follows that of Lemma~\ref{lem:rec}, but here we omit the proof.  
\end{proof}

\section{Two Examples of Growing Grids}
 This section shows two examples of random walks on growing integer grids. 
 Section~\ref{sec:box} briefly shows the phase transition 
  between the recurrence and transience 
  of a simple random on a growing box of $\mathbb{Z}^d$ with a fixed $d$, 
 where essentially the same fact is given by Dembo et al.~\cite{DHS14J} 
  except for some small differences 
   e.g., the model of a growing graph, the period of the random walk, and the proof technique. 
Section~\ref{sec:cube} shows an interesting example, 
  where we show the phase transition 
%  between the recurrence and transience 
  of a simple random on the $\{0,1\}^n$ with an increasing $n$. 
\subsection{Random walk on growing box}\label{sec:box}
\subsubsection{Result summary}
 Let $B_n = (V_n,E_n)$ denote the integer grid graph of $\mathbb{Z}^d$ bounded by $-n$ and $n$, 
  which is formally given by 
\begin{align*}
V_n &= \prod_{i=1}^d\{-n,\ldots,-1,0,1,\ldots,n\} \\
E_n & \textstyle = \left\{ \{\Vec{x},\Vec{y}\} \in {V_n \choose 2} \mid \|\Vec{x}-\Vec{y}\|_1=1   \right\}
\end{align*}
 for $n=1,2,\ldots$. 
We consider a simple random walk with the reflection bound, that is formally given by 
\begin{equation} 
P_n(\Vec{u},\Vec{v}) = 
\begin{cases}
\frac{1}{2d} & \text{if $\|\Vec{u}-\Vec{v}\|_1=1$, $u_i \neq v_i$ and $|u_i| \neq n$} \\
\frac{1}{d} & \text{if $\|\Vec{u}-\Vec{v}\|_1=1$, $u_i \neq v_i$ and $|u_i| = n$,} \\
0 & \text{otherwise,}
\end{cases}
\end{equation}
   for $\Vec{u} = (u_1,\ldots,u_d)  \in V_n$ and $\Vec{v}=(v_1,\ldots,v_d) \in V_n$. 
The following fact is well known. 
\begin{proposition}[cf.~\cite{Durrett}]
$\Vec{0}$ is recurrent by $P_{\infty}$ when $d \leq 2$. 
$\Vec{0}$ is transient by $P_{\infty}$ when $d \geq 3$. 
\end{proposition}
 Then, we are concerned with a RWaGG ${\cal D}_{\rm B} = (\mathfrak{d},G,P)$ where $G(n)=B_n$ and $P(n) = P_n$. 
 We remark that the dimension $d$ is \emph{constant} to the size of box $n$. 
\begin{theorem}\label{thm:box}
Suppose $d \geq 4$. Then, 
$\Vec{0}$ is recurrent by $\mathcal{D}_{\mathrm{B}}$
if 
\begin{align}
 \sum_{n=1}^{\infty}\frac{\mathfrak{d}(n) }{n^d}=\infty ,
\end{align}
 holds, otherwise $\Vec{0}$ is transient. 
\end{theorem}
\begin{corollary}\label{cor:box}
If $\mathfrak{d}(n) = \Omega(n^{d-1})$ then $\Vec{0}$ is recurrent by $\mathcal{D}_{\mathrm{B}}$. 
If $\mathfrak{d}(n) = \Order(n^{d-1-\epsilon})$ then $\Vec{0}$ is transient by $\mathcal{D}_{\mathrm{B}}$. 
\end{corollary}

\subsubsection{Proof of Theorem~\ref{thm:box}}
% For convenience, let $h_i(\Vec{v}) = |v_i|$ for $\Vec{v} = (v_1,\ldots,v_d)$, for $i=1,2,\ldots,d$.  
 To begin with, we remark the following fact. 
\begin{lemma}\label{lem:pn-box}
(i) Every $P_n$ ($n=1,2,\ldots$) is reversible: precisely, let 
\begin{align}
 \phi(\Vec{v}) = 2^{-\left|\{ i \in \{1,\ldots,d\}  \mid |v_i| = n \}\right| }
\label{eq:reversible-box}
\end{align}
 for $\Vec{v} = (v_1,\ldots,v_d) \in V_n$, 
  then the detailed balance equation 
\begin{align}
 \phi(\Vec{u})P_n(\Vec{u},\Vec{v}) = \phi(\Vec{v})P_n(\Vec{v},\Vec{u})
\label{dbe:box}
\end{align}
 holds for any $\Vec{u},\Vec{v} \in V_n$. 
(ii) Every $P_n$ is irreducible and ${\rm period}(P_n) = 2$. 
 Thus the even-time stationary distribution of $P_n$ is 
\begin{align}
\mathring{\pi}_n(\Vec{v}) 
= \frac{\phi(\Vec{v})}{\sum_{\Vec{u} \in U_n} \phi(\Vec{u})}
\label{eq:stationary-box}
\end{align}
for any $\Vec{v} \in U_n$. 
%(iii) $R(t) = 0$ for any odd $t$. 
\end{lemma}

We do not give a precise value of $\mathring{\pi}_n(\Vec{0})$, but the following lemma shows its upper and lower bounds.  
\begin{lemma}\label{lem:pn-box-bound}
\begin{align*}
\frac{1}{(2n+1)^d}
\leq p(n) 
\leq \frac{2}{(2n-1)^d}
\end{align*}
holds for any $n \in \mathbb{Z}_{>0}$. 
%\begin{align}
%\frac{1}{\prod_{i=1}^d \left(2\lfloor \frac{b_i(n)}{2}\rfloor+1\right)}
%\leq p(n) 
%\leq \frac{1}{\prod_{i=1}^d \max\left\{1,\left(2 \lfloor \frac{b_i(n)}{2}\rfloor -1\right)\right\}}
%\end{align}
\end{lemma}

\begin{lemma}\label{lem:mixing-box}
Let $\mathring{\tau}_n(\epsilon)$ denote the mixing time of $P_n$, then 
$\mathring{\tau}_n(\epsilon) = O(n^2 d \log d/\epsilon)$. 
\end{lemma}

\begin{lemma}\label{lem:lhagg-box}
$\mathcal{D}_{\rm B}$ is LHaGG. %, i.e., $\Pr[X_t=r] \geq \Pr[Y_t=r]$. 
\end{lemma}
\begin{proof}
 Let $f$ and $g$ satisfy $\sum_{i=1}^n f(i) \leq \sum_{i=1}^n g(i)$ for any $n=1,2,\ldots$, and 
  let $X_t$ and $Y_t$ ($t=0,1,2,\ldots$) respectively follow $(f,G,P)$ and $(g,G,P)$, 
  i.e., the box of $(f,G,P)$ grows faster than $(g,G,P)$. 
 Let $X_0 =Y_0 =\Vec{0}$, and 
  we prove $\Pr[X_t = \Vec{0}] \leq \Pr[Y_t = \Vec{0}]$ for any $t=1,2,\ldots$. % (recall Section \ref{sec:lhagg} for LHaGG). 

 The proof is (a kind of) coupling, and 
  we prove $|X_{t,i}| \geq |Y_{t,i}|$ for any $i=1,2,\ldots,d$ at any time $t=1,2,\ldots$. 
 Consider the pair of transitions $X$ to $X'$ and $Y$ to $Y'$, for convenience of the notation.  
 In any transition, let both $X$ and $Y$ choose the same index $i \in \{1,\ldots,d\}$ for the transition. 
 We remark that $|X_i| > |Y_i|$ implies $|X_i| \geq |Y_i|+2$, 
   because $|X'_i - Y'_i|$ should be one of $|X_i - Y_i|+2$, $|X_i - Y_i|$ or $|X_i - Y_i|-2$. 
 Thus, if $|X_i| > |Y_i|$ then $|X'_i| \geq |X_i|-1 \geq |Y_i|+1  \geq  |Y'_i|$ holds, meaning that $|X'_i| \geq |Y'_i|$. 
 
Then, suppose $|X_i| = |Y_i|$. 
We consider four cases: 
(i) $Y_i=0$,  
(ii) $0 <  |Y_i| < n_y$, 
(iii) $|Y_i|=n_Y$ and $n_Y < n_X$, 
(iv) $|Y_i| =n_Y= n_X$, 
where $n_X$ and $n_Y$ respectively denotes the phase (or, boundary of the boxes) of $\mathcal{D}$ and  $\mathcal{D}'$ at time $t$. 
%%%
 In case (i), 
\begin{align*}
\Pr[|X'_i| = |X_i|+1] &= \Pr[Y'_i = Y_i+1] = \frac{1}{d}
\end{align*}
 hold, and hence we can couple them to satisfy $|X'_i| = |Y'_i|$. % and hence $h(X') \geq h(Y')$. 
%%%
 In case (ii), 
\begin{align*}
\Pr[|X'_i| = |X_i|+1] &= \Pr[|Y'_i| = |Y_i|+1] = \frac{1}{2d} \\
\Pr[|X'_i| = |X_i|-1] &= \Pr[|Y'_i| = |Y_i|-1] = \frac{1}{2d} 
\end{align*}
 hold, and hence we can couple them to satisfy $|X'_i| = |Y'_i|$. % and hence $h(X') \geq h(Y')$. 
%%%%
 In case (iii), 
\begin{align*}
\Pr[|Y'_i| = |Y_i|-1] = 1
\end{align*}
 holds, and hence we can couple them to satisfy $|X'_i| \geq |Y'_i|$. % and hence $h(X') \geq h(Y')$. 
 In case (iv), 
\begin{align*}
\Pr[|X'_i| = |X_i|-1] = \Pr[|Y'_i| = |Y_i|-1] = 1
\end{align*}
 holds, and hence we can couple them to satisfy $|X'_i| = |Y'_i|$. 

 Thus, we got a coupling to satisfy $|X_{t,i}| \geq |Y_{t,i}|$ for any $i=1,2,\ldots, d$ at any time $t=1,2,\ldots$. 
 It implies that if $X_t = \Vec{0}$ then $Y_t = \Vec{0}$, and hence that $\Pr[X_t =\Vec{0}] \leq \Pr[Y_t=\Vec{0}]$. 
 We obtain the claim. 
\end{proof}

\begin{proof}[Proof of Theorem~\ref{thm:box}]
%First we prove the case that $|V_n|$ grows enough large such that 
%$\sum_{n=1}^{\infty}\frac{1}{|V_n|}<\infty$, 
%e.g., $|V_n| = \Omega (n^{1+\epsilon})$. 

(Recurrence) 
By Lemma~\ref{lem:lhagg-box}, ${\cal D}_{\rm B}$ is LHaGG. 
Since $p(n) \simeq \frac{1}{2^d}\frac{1}{n^d}$ by Lemma~\ref{lem:pn-box-bound}, 
Lemma~\ref{lem:rec} implies that 
if $\sum_{n=1}^{\infty}\frac{\mathfrak{d}(n) }{n^d}=\infty$ then $\Vec{0}$ is recurrent.

(Transience)
By Lemma~\ref{lem:mixing-box}, 
$\mathring{\mathfrak{t}}(n) = \mathring{\mathfrak{\tau}}_n (p(n-1)) \leq n^2 \log \frac{1}{p(n-1)} \leq n^2 \log ((2n-1)^d) \leq c' n^2 \log n$, and hence 
$\sum_{n=1}^{\infty} \mathring{\mathfrak{t}}(n) p(n-1) \leq  \sum_{n=1}^{\infty} c'n^2 \log n \frac{2}{(2n-1)^d} < \infty$ when $d\geq 4$. 
If $\sum_{n=1}^{\infty} \frac{\mathfrak{d}(n)}{n^d} < \infty$, 
 then $\sum_{n=1}^{\infty} \max\{ \mathfrak{d}(n),  \mathring{\mathfrak{t}}(n)\} p(n-1) \leq \sum_{n=1}^{\infty} (\mathfrak{d}(n) + \mathring{\mathfrak{t}}(n)) \frac{1}{(2n-3)^d} \leq \frac{2}{2^d}\sum_{n=1}^{\infty} (\mathfrak{d}(n) + \mathring{\mathfrak{t}}(n)) \frac{1}{n^d} < \infty$, 
 which implies $\sum_{t=1}^{\infty} R(t) < \infty$ by Lemma~\ref{lem:trans} with Lemma~\ref{lem:lhagg-box}. 
\end{proof}

\subsubsection{Remark}
% Coupling on integer lattice is a classical topic in the mixing time of Markov chains, and well studied.  
% We can obtain Theorem~\ref{?} by Dembo et al.~\cite{?}, 
%  while we remark that our example is periodic. 

%On the other hand, we have some small trouble that 
% if $V_n = \prod_i^d \{-a_i(n),\ldots,b_i(n)\}$ 
% by the coupling argument. 

 The growing box model is slightly generalized, 
   while our proof may remain a gap between the recurrence and transition 
    depending on the parameter $b_i$.  
 Let $b_i \colon \mathbb{Z}_{>0} \to  \mathbb{Z}_{>0} $ for $i=1,2,\ldots,d$  
  be monotone non-decreasing functions with respect to $n$. 
% For convenience of a proof (Lemma~\ref{lem:pn-box-bound}), 
%   we assume $b_i(n) \geq 2$ for any $n$, but it is not essential. 
Let $B_n = (V_n,E_n)$ where 
\begin{align*}
V'_n &= \prod_{i=1}^d\{-b_i(n),\ldots,-1,0,1,\ldots,b_i(n)\} \\
E'_n & \textstyle = \left\{ \{\Vec{x},\Vec{y}\} \in {V_n \choose 2} \mid \|\Vec{x}-\Vec{y}\|_1=1   \right\}
\end{align*}
 for $n=1,2,\ldots,$ where we remark $d$ is constant to $n$. 
We consider the reflection bound, that is formally given by 
\begin{align*} 
P'_n(\Vec{u},\Vec{v}) = 
\begin{cases}
\frac{1}{2d} & \text{if $\|\Vec{u}-\Vec{v}\|_1=1$, $u_i \neq v_i$ and $|u_i| \neq b_i(n)$} \\
\frac{1}{d} & \text{if $\|\Vec{u}-\Vec{v}\|_1=1$, $u_i \neq v_i$ and $|u_i| = b_i(n)$,} \\
0 & \text{otherwise,}
\end{cases}
\end{align*}
   for $\Vec{u} = (u_1,\ldots,u_d)  \in V_n$ and $\Vec{v}=(v_1,\ldots,v_d) \in V_n$. 
 Then, we are concerned with RWoGG $\mathcal{D}_{\mathrm{B}'} = (\mathfrak{d},G',P')$ 
 where $G'(n) = B_n$ and $P'(n) = P'_n$. 
\begin{theorem}\label{thm:box2}
$\Vec{0}$ is recurrent by $\mathcal{D}_{\mathrm{B}'}$
if 
\begin{align*}
 \sum_{n=1}^{\infty}\frac{\mathfrak{d}(n) -1 }{|V_n|}=\infty 
\end{align*}
 holds. $\Vec{0}$ is transient by $\mathcal{D}_{\mathrm{B}'}$ if  
\begin{align*}
 \sum_{n=1}^{\infty}\frac{\max\{\mathfrak{d}(n), \sum_{i=1}^n b_i(n)^2 \log  |V_n| \} }{|V_n|}<\infty 
\end{align*}
holds. 
\end{theorem}
The proof is essentially the same as Theorem~\ref{thm:box}. 
We remark that $b_i(n)$ can be constant (for some $i$, otherwise trivially recurrent), 
 that allows the box not to be fat, cf.\ Assumption 1.1 of \cite{DHS14J}. 

\subsection{Random walk on $\{0,1\}^n$ with an increasing $n$}\label{sec:cube} 
\subsubsection{Main result}
This section shows an interesting example. 
Let $C_n = (V_n,E_n)$ where 
\begin{align}
V_n &= \{0,1\}^n \\
E_n & \textstyle = \left\{ \{\Vec{u},\Vec{v}\} \in {V_n \choose 2} \mid \|\Vec{u}-\Vec{v}\|_1=1   \right\}
\end{align}
for $n=1,2,\ldots$. 
Let $\Vec{0} \in V_n$ denote the (common) origin vertex $(0,\ldots,0)$ for each $n$. 
Let 
\begin{equation} 
P_n(\Vec{u},\Vec{v}) = 
\begin{cases}
\frac{1}{n} & \text{if $\|\Vec{u}-\Vec{v}\|_1=1$,} \\
0 & \text{otherwise,}
\end{cases}
\end{equation}
for $\Vec{u},\Vec{v} \in V_n$. 
Then, we are concerned with ${\cal D}_{\rm C} = (\mathfrak{d},G,P)$ starting from $\Vec{0}$ where $G(n)=C_n$ and $P(n) = P_n$. 
\begin{theorem}\label{thm:cube}
If $\mathcal{D}_{\rm C}$ satisfies
\begin{align}
 \sum_{n=1}^{\infty}\frac{\mathfrak{d}(n) }{2^{n}}=\infty 
\end{align}
then $\Vec{0}$ is recurrent, otherwise $\Vec{0}$ is transient. 
\end{theorem}

The following lemma is not very difficult, but nontrivial. 
\begin{lemma}\label{lem:lhagg-cube}
${\cal D}_{\rm C}$ is LHaGG. 
\end{lemma}
\begin{proof}
 Let $f$ and $g$ satisfy $\sum_{i=1}^n f(i) \leq \sum_{i=1}^n g(i)$ for any $n=1,2,\ldots$, and 
  let $X_t$ and $Y_t$ ($t=0,1,2,\ldots$) respectively follow $(f,G,P)$ and $(g,G,P)$, 
  i.e., the box of $(f,G,P)$ grows faster than $(g,G,P)$. 
 Let $n_t$ (resp.\ $n'_t$) denote the dimension of $(f,G,P)$ (resp.\ $(g,G,P)$) at time $t$, 
  and then notice that $n_t \geq n_t'$ hold for any $t=0,1,\ldots$ by the assumption that $(f,G,P)$ grows faster. 
 Let $X_0 =Y_0 =\Vec{0}$, and 
  we prove $\Pr[X_t = \Vec{0}] \leq \Pr[Y_t = \Vec{0}]$ for any $t=1,2,\ldots$. 

 Let $h(\Vec{u}) = |\{ i \in \{1,\ldots,n\} \mid u_i =1\}|$ for $\Vec{u}=(u_1,\ldots,u_n) \in V_n$. 
 We construct a coupling of $\Vec{X} = \{X_t\}_{t \geq 0}$ and  $\Vec{Y} = \{Y_t\}_{t \geq 0}$ 
  such that 
   $h(X_t) \geq h(Y_t)$ holds for any $t=1,2,\ldots$. 
 The proof is an induction concerning $t$. 
 Clearly, $h(X_0) = h(Y_0)=0$. 
 Inductively assuming $h(X_t) \geq h(Y_t)$, we prove  $h(X_{t+1}) \geq h(Y_{t+1})$.  
 If $h(X_t) > h(Y_t)$ then $h(X_t) \geq h(Y_t)-2$ 
  since every $P_n$ is ${\rm period}(P_n)=2$ for $n=1,2,\ldots$. 
 It is easy to see that 
  $h(X_{t+1}) \geq h(X_t)-1 \geq h(Y_t)+1 \geq h(Y_{t+1})$, and we obtain $h(X_{t+1}) \geq h(Y_{t+1})$ in the case. 
 Suppose $h(X_t)=h(Y_t)$. 
 Then, 
\begin{align*}
\Pr[h(X_{t+1}) = h(X_t)-1] &= \frac{h(X_t)}{n_t} \leq \frac{h(Y_t)}{n'_t} =  \Pr[h(Y_{t+1}) = h(Y_t)-1], \hspace{4em} \mbox{and} \\
\Pr[h(X_{t+1}) = h(X_t)+1] &= 1-\frac{h(X_t)}{n_t} \geq 1-\frac{h(Y_t)}{n'_t} = \Pr[h(Y_{t+1}) = h(Y_t)+1]
\end{align*}
 hold, which implies that a coupling exists such that $h(X_{t+1}) \geq h(Y_{t+1})$. 

 Now we obtain a coupling of $\Vec{X} = \{X_t\}_{t \geq 0}$ and  $\Vec{Y} = \{Y_t\}_{t \geq 0}$ satisfying $h(X_t) \geq h(Y_t)$ for any $t=1,2,\ldots$, 
  which implies that $h(Y_t)=0$ as long as $h(X_t) = 0$. 
 This means that $\Pr[X_t = \Vec{0}] \leq \Pr[Y_t= \Vec{0}]$ for any $t=1,2,\ldots$.  
 We obtain the claim. 
\end{proof}

The following two lemmas are well known. 
\begin{lemma}\label{lem:mixing-cube}
Let $\mathring{\tau}_n(\epsilon)$ denote the mixing time of $P_n$. 
Then, $\mathring{\tau}_n(\epsilon) =\Order( n\log (n/\epsilon))$. 
\end{lemma}
\begin{lemma}\label{lem:pn-cube}
$p(n) = \frac{1}{\frac{2^n}{2}} = 2^{-n+1}$.
\end{lemma}

\begin{proof}[Proof of Theorem~\ref{thm:cube}]

(Recurrence) 
By Lemma~\ref{lem:lhagg-cube}, ${\cal D}_{\rm C}$ is LHaGG. 
 Since $p(n) = 2^{-n+1}$ by Lemma~\ref{lem:pn-cube}, 
Lemma~\ref{lem:rec} implies that 
if $\sum_{n=1}^{\infty}\frac{\mathfrak{d}(n)}{2^n}=\infty$ then $\Vec{0}$ is recurrent. 

(Transience)
By Lemma~\ref{lem:mixing-cube}, 
$\mathring{\mathfrak{t}}(n) = \mathring{\mathfrak{\tau}}_n (p(n-1)) \leq n \log \frac{n}{p(n-1)} \leq n \log (n 2^n) \leq c' n^2 \log n$, and hence 
$\sum_{n=1}^{\infty} \mathring{\mathfrak{t}}(n) p(n-1) \leq  \sum_{n=1}^{\infty} c'n^2 \log n \frac{1}{2^{n-2}} < \infty$. 
If $\sum_{n=1}^{\infty} \frac{\mathfrak{d}(n)}{2^n} < \infty$, 
 then $\sum_{n=1}^{\infty} \max\{ \mathfrak{d}(n),  \mathring{\mathfrak{t}}(n)\} p(n-1) \leq \sum_{n=1}^{\infty} (\mathfrak{d}(n) + \mathring{\mathfrak{t}}(n)) \frac{1}{2^n}  < \infty$, 
 which implies $\sum_{t=1}^{\infty} R(t) < \infty$ by Lemma~\ref{lem:trans} with Lemma~\ref{lem:lhagg-cube}. 
\end{proof}

\subsubsection{An Interesting fact: every  finite point becomes recurrent}
 We can easily observe the following fact from Theorem~\ref{thm:cube}. 
\begin{corollary}\label{cor:cube}
  If $\mathfrak{d}(n) = \Omega(2^n/n)$ then $\Vec{0}$ is recurrent. 
  If $\mathfrak{d}(n) = \Order(2^n/n^{1+\epsilon})$ then $\Vec{0}$ is transient. 
\end{corollary}
 Notice that the maximum degree of $G(n)$ is unbounded asymptotic to $n$, clearly. 
 Nevertheless, 
  we can see the following interesting fact. 
\begin{proposition}\label{prop:cube-everyv}
 If $\mathfrak{d}(n) = \Omega(n 2^n )$ 
  then $\mathcal{D}_{\rm C}$ starting from $\Vec{0}$ visits $\Vec{v} \in V_m$ infinitely many times for any $m<\infty$. 
\end{proposition}
\begin{proof}
 Notice that $\mathring{\tau}_n(2^{-n-1}) = \Order (n \log (n 2^{n+1})) = \Order(n^2 \log n)$ by Lemma~\ref{lem:mixing-cube}. 
 Thus, in the $n$-th phase, i.e, $[T_{n-1},T_n]$, 
  the random walk $X$ visits $\Vec{v} \in V_n$ with probability at least $2^{-n-1}$ in every $\Order(n^2 \log n)$ steps 
   (even if $\Vec{v} \in \overline{U}_n$, here we omit the proof). 
 Thus the probability that $X$ never visit $\Vec{v}$ during the $n$-th phase is at most 
  $(1-2^{-n-1})^{n 2^{n+1}/n^2 \log n} \leq \exp(-\frac{1}{n \log n})$. 
 This implies that the probability that $X$ never visits $\Vec{v} \in V_m$ forever is at most 
  $\prod_m^{\infty} \exp(-\frac{1}{n \log n}) 
    = \exp(- \sum_{n=m}^{\infty} \frac{1}{n \log n}) 
    \leq \exp (- \int_m^{\infty} \frac{1}{x \log x} {\rm d}x) 
    = \exp(-[\log \log x]_m^{\infty}) = \exp(-\infty) = 0$. 
 This means that the RWoGG $X$ visits $\Vec{v} \in V_m$ at least once in finite steps with probability 1. 

 Once we know that $X$ visits $\Vec{v}$ in a finite steps, 
  the claim is trivial thanks to the vertex transitivity of the hypercube skeleton. 
\end{proof}
 We think that the hypothesis of Proposition~\ref{prop:cube-everyv} 
   can be relaxed from $\Omega(n 2^n)$ to $\Omega(2^n/n)$, but we are not sure.

\section{Simple RW on a Growing Tree with Unbounded Degrees}\label{sec:tree}
%Some readers may feel complete $k$-ary tree is too specific. 
This section shows some more easy examples  of random walks on growing trees being LHaGG. 

\subsection{Simple random walk growing level trees}
  Let $\tree_n = (V_n,E_n)$ be a tree for $n=1,2,\ldots$, 
    where we assume $V_n \subset V_{n+1}$  and $E_n \subset E_{n+1}$. 
 We call $\tree_n$  a {\em level tree} if 
%  (i) 
   $\deg_n(u) = \deg_n(v)$ holds whenever $h(u) = h(v)$. 
We remark the condition implies    
%  (ii) 
    all leaves place the same height 
      because $\deg_n(v)=1$ if and only if $v$ is a leaf.
 We remark that $\deg_n(v)$ can be unbounded asymptotic to $n$.  
 Let $r \in V_n$ denote the root, that is the unique vertex of height 0.  
 For convenience, let $h(v)$ denote the height of vertex $v \in V_n$, 

%%%%
The transition probability is given by 
\begin{align} 
P_n(u,v) = %\frac{1}{\deg_n(u)}
\begin{cases}
 \gamma & \text{if $v =v$,} \\
 (1-\gamma)\dfrac{1}{\deg_n(u)} &\text{if $\{u,v\} \in E$,} \\
 0 & \text{otherwise,}
\end{cases}
%\label{eq:comp-tree1}
\end{align}
   for $u,v \in V_n$. 
 We say the random walk is \emph{lazy} when $\gamma > 0$, otherwise, i.e., $\gamma = 0$, \emph{busy}. 
 Clearly, a lazy random walk is aperiodic, while a busy random walk has period 2. 
 We are concerned with ${\cal D}_{\rm T'} = (\mathfrak{d},G,P)$, 
 where $G(n)=G_n$ and $P(n) = P_n$, 
  considering cases $\gamma = 0$ or not. 

\subsection{``Busy'' walk}
Here, we are concerned with the case $\gamma=0$. 
\begin{lemma}
If $\gamma = 0$ 
then $\mathcal{D}_{\mathrm{T}'}$ is LHaGG. 
\end{lemma}\label{lem:lhagg-leveltree-busy}
 The proof is almost the same as Lemma~\ref{lem:lhagg-karytree} for the complete $k$-ary tree. 
\begin{proof}
 Let $f$ and $g$ satisfy $\sum_{i=1}^n f(i) \leq \sum_{i=1}^n g(i)$ for any $n=1,2,\ldots$, and 
  let $X_t$ and $Y_t$ ($t=0,1,2,\ldots$) respectively follow $(f,G,P)$ and $(g,G,P)$, 
  i.e., the tree of $(f,G,P)$ grows faster than $(g,G,P)$. 
 Let $X_0 =Y_0 =r$, and 
  we prove $\Pr[X_t = r] \leq \Pr[Y_t = r]$ for any $t=1,2,\ldots$. 

 The proof is (a kind of) coupling, and 
  we prove $h(X_t) \geq h(Y_t)$ for any $t=1,2,\ldots$. 
 Consider the pair of transitions $X$ to $X'$ and $Y$ to $Y'$, for convenience of the notation.  
To begin with, we remark that $h(X) > h(Y)$ implies $h(X) \geq h(Y)-2$ 
  since every $P_n$ is ${\rm period}(P_n)=2$ for $n=1,2,\ldots$. 
Thus, $h(X') \geq h(X)-1 \geq h(Y)+1 \geq h(Y')$ 
  holds by any transition, in the case. 

Consider the case $h(X)=h(Y)$. We consider four cases: 
(i) $X=Y=r$,  
(ii) both $X$ and $Y$ are inner nodes, 
(iii) both $X$ and $Y$ are leaves, i.e., both trees of $(f,G,P)$ and$(g,G,P)$ take the same height at time $t$, 
(iv) $X$ is not a leaf but $Y$ is a leaf, i.e., the tree of $(f,G,P)$ is higher than that of $(g,G,P)$ at time $t$. 

%%%
 In case (i), 
\begin{align*}
\Pr[h(X') = h(X)+1] &= \Pr[h(Y') = h(Y)+1] = 1
\end{align*}
 hold, and hence we can couple them to satisfy $h(X') = h(Y')$. % and hence $h(X') \geq h(Y')$. 
 In case (ii), since both $X$ and $Y$ are inner nodes, 
\begin{align*}
\Pr[h(X') = h(X)-1] &=\tfrac{1}{\deg_n(X)} \leq \tfrac{1}{\deg_{n'}(Y)}  = \Pr[h(Y') = h(Y)-1]  \\
\Pr[h(X') = h(X)+1] &=1-\tfrac{1}{\deg_n(X)}\geq 1-\tfrac{1}{\deg_{n'}(Y)} = \Pr[h(Y') = h(Y)+1] 
\end{align*}
 hold since $\deg_n(X) \geq \deg_{n'}(Y)$, 
   and hence we can couple them to satisfy $h(X') = h(Y')$. % and hence $h(X') \geq h(Y')$. 
%%%%
 In case (iii), both $X$ and $Y$ are leaves, 
\begin{align*}
\Pr[h(X') = h(X)-1] &= \Pr[h(Y') = h(Y)-1] = 1
\end{align*}
 holds, and hence we can couple them to satisfy $h(X') = h(Y')$. % and hence $h(X') \geq h(Y')$. 
 In case (iv), 
  since $X$ is not a leaf but $Y$ is a leaf, 
\begin{align*}
\Pr[h(X') = h(X)-1] &= \tfrac{1}{\deg_n(X)} \leq \Pr[h(Y') = h(Y)-1] =1  
\end{align*}
 holds, and hence we can couple them to satisfy $h(X') \geq h(Y')$. 
 Thus, we got a coupling to satisfy $h(X_t) \geq h(Y_t)$ for any $t=1,2,\ldots$. 
 It implies that if $h(X_t) = 0$ then $h(Y_t)=0$, and hence that $\Pr[X_t =r] \leq \Pr[Y_t=r]$. 
 We obtain the claim. 
\end{proof}

\begin{proposition}\label{prop:level-busy}
If $\mathfrak{d}$ satisfies 
\begin{align*}
 \sum_{n=1}^{\infty}\frac{\mathfrak{d}(n) -1 }{|E_n|}=\infty 
\end{align*}
then $r$ is recurrent by $\mathcal{D}_{\mathrm{T}'}$. 
\end{proposition}

\subsection{Lazy walk}
Next, we are concerned with the case $\gamma \neq 0$. 
\begin{lemma}
 If $\frac{1}{2} \leq \gamma < 1$ 
  then $\mathcal{D}_{\mathrm{T}'}$ is LHaGG. 
\end{lemma}
\begin{proof}
 Let $f$ and $g$ satisfy $\sum_{i=1}^n f(i) \leq \sum_{i=1}^n g(i)$ for any $n=1,2,\ldots$, and 
  let $X_t$ and $Y_t$ ($t=0,1,2,\ldots$) respectively follow $(f,G,P)$ and $(g,G,P)$, 
  i.e., the tree of $(f,G,P)$ grows faster than $(g,G,P)$. 
 Let $X_0 =Y_0 =r$, and 
  we prove $\Pr[X_t = r] \leq \Pr[Y_t = r]$ for any $t=1,2,\ldots$. 
% The proof is based on a (type of) coupling. 

%We prove the claim by coupling. 
%Suppose $X_t$ and $Y_t$ respectively follows ${\cal D}$ and  ${\cal D'}$, where $X_0=Y_0=r$. 
%Let $h(v)$ denote the height of $v \in V$, i.e., the length from $v$ to $r$. 
%Note that $h(r)=0$. 
 We prove $h(X_t) \geq h(Y_t)$ for any $t=1,2,\ldots$. 
 Consider the pair of transitions $X$ to $X'$ and $Y$ to $Y'$, for convenience of the notation.  
%If $h(X') < h(X)$ then $h(Y') < h(Y)$
 If $h(X) \geq h(Y)+2$ then $h(X') \geq h(Y')$ is trivial, 
  similar to Lemma~\ref{lem:lhagg-leveltree-busy}
 Similarly, 
  the case $h(X) =h(Y)$ is similar to Lemma~\ref{lem:lhagg-leveltree-busy}. 

Then, 
  we consider the case $h(X)=h(Y)+1$. We consider six cases: 
(i) $Y=r$ and $X$ is a leaf, i.e., $n=1$,  
(ii) $Y=r$ and $X$ is not a leaf, i.e., $n \geq 2$, 
(iii) both $Y$ and $X$ are inner nodes, 
(iv) $Y$ is a leaf but $X$ is not a leaf, 
   i.e., the tree of $(f,G,P)$ is higher more than one than that of $(g,G,P)$ at time $t$. 
(v) both $Y$ and $X$ are leaves, 
   i.e., both trees of the height of the tree $(f,G,P)$ is greater exactly one than that of $(g,G,P)$ at time $t$, 
(vi) $Y$ is not a leaf but $X$ is a leaf, 
   i.e., the tree of $(f,G,P)$ is the same height as that of $(g,G,P)$ at time $t$. 

%%%
 In case (i), 
\begin{align*}
\Pr[h(X') = h(X)-1] &= 1-\gamma
\leq \Pr[h(Y') = h(Y)] = \gamma
\end{align*}
 hold, where the inequality follows $\gamma \geq \frac{1}{2}$. 
 Thus, we can couple them to satisfy $h(X') \geq h(Y')$. % and hence $h(X') \geq h(Y')$. 
 In case (ii), 
\begin{align*}
\Pr[h(X') = h(X)-1] &= (1-\gamma)\tfrac{1}{\deg_n(X)} \leq \Pr[h(Y') = h(Y)] = \gamma,\hspace{2em} \mbox{and}\\
\Pr[h(Y') = h(Y)+1] &= (1-\gamma)\left(1-\tfrac{1}{\deg_{n'}(Y)}\right) \leq  \Pr[h(X') = h(X)] = \gamma
\end{align*}
 hold since $\gamma \geq \frac{1}{2}$. 
 Thus, we can couple them to satisfy $h(X') \geq h(Y')$. % and hence $h(X') \geq h(Y')$. 
 In case (iii), similarly 
\begin{align*}
\Pr[h(X') = h(X)-1] &= (1-\gamma)\tfrac{1}{\deg_n(X)} \leq \Pr[h(Y') = h(Y)] = \gamma,\hspace{2em} \mbox{and}\\
\Pr[h(Y') = h(Y)+1] &= (1-\gamma)\left(1-\tfrac{1}{\deg_{n'}(Y)}\right) \leq  \Pr[h(X') = h(X)] = \gamma
\end{align*}
 hold, and hence we can couple them to satisfy $h(X') \geq h(Y')$. % and hence $h(X') \geq h(Y')$. 
%%%%
 In case (iv), since $Y$ is a leaf, 
\begin{align*}
\Pr[h(X') = h(X)-1] = (1-\gamma) \left(1-\tfrac{1}{\deg_n(X)}\right)  \leq \Pr[h(Y') = h(Y)] = \gamma
\end{align*}
 and hence we can couple them to satisfy $h(X') \geq h(Y')$. % and hence $h(X') \geq h(Y')$. 
 In case (v), 
  since both $X$ and $Y$ are leaves, 
\begin{align*}
 \Pr[h(X') = h(X)-1] = \Pr[h(Y') = h(Y)-1] =  1-\gamma
\end{align*}
 hold and hence we can couple them to satisfy $h(X') \geq h(Y')$. % and hence $h(X') \geq h(Y')$. 
 In case (vi), 
  since $X$ is a leaf and $Y$ is not a leaf
\begin{align*}
\Pr[h(X') = h(X)-1] &= (1-\gamma)\tfrac{1}{\deg_n(X)} \leq \Pr[h(Y') = h(Y)] = \gamma,\hspace{2em} \mbox{and}\\
\Pr[h(Y') = h(Y)+1] &= (1-\gamma)\left(1-\tfrac{1}{\deg_{n'}(Y)}\right) \leq  \Pr[h(X') = h(X)] = \gamma
\end{align*}
 hold, and hence we can couple them to satisfy $h(X') \geq h(Y')$. 
 Thus, we got a coupling to satisfy $h(X_t) \geq h(Y_t)$ for any $t=1,2,\ldots$. 
 It implies that if $h(X_t) = 0$ then $h(Y_t)=0$, and hence that $\Pr[X_t =r] \leq \Pr[Y_t=r]$. 
 We obtain the claim. 
\end{proof}

\begin{proposition}\label{prop:level-lazy}
If $\mathfrak{d}$ satisfies 
\begin{align*}
 \sum_{n=1}^{\infty}\frac{\mathfrak{d}(n) }{|E_n|}=\infty 
\end{align*}
then $r$ is recurrent by $\mathcal{D}_{\mathrm{T}'}$. 
\end{proposition}

  A similar result for $\gamma$ satisfying $0 < \gamma \leq 1/2$ 
   is given by  Huang~\cite{Huang19} 
   under some appropriate assumptions, e.g., uniformly bounded degree. 
 
\subsection{Growing Star for Unbounded Degree}
%\subsubsection{Remark on RW on growing level trees}
We remark that the degree of the growing level tree can be unbounded asymptotic to $t \to \infty$, 
so does the simple random walk on $\{0,1\}^n$ in Section~\ref{sec:cube}.

As an extreme case of growing level tree, 
  we consider growing star, formally given by  
$G_n = (V_n,E_n)$, $V_n = \{r,v_1,\ldots,v_{M(n)}\}$ and $E_n = \{\{r,v_i\} \mid i=1,\ldots,M(n)\}$, 
where $M \colon \mathbb{Z}_{>0} \to  \mathbb{Z}_{>0}$ be a monotone function. 
It is not difficult to see that $r$ is recurrent for any $\mathfrak{d}$. 
On the other hand, 
  $v_1$ ($v_i$ as well for finite $i$) is recurrent depending on  $\mathfrak{d}$, 
 where we remark that $\mathfrak{d}(n) = 0$ for some $n$ is allowed, 
  unlike Section~\ref{sec:cube}. 
Let $P_n$ denote the busy simple walk, i.e., $P(u,v) = 1/\deg_n(u)$ for $\{u,v\} \in E$, 
 while let $P'_n$ denote the lazy simple walk, 
  i.e., $P(u,v) = (1-\gamma)/\deg_n(u)$ for $\{u,v\} \in E$ and $P(u,u) = \gamma$ for all $u \in V_n$, 
Let ${\mathcal D}_{\rm S} = (\mathfrak{d},P,G)$ denote the busy RWoGG, and 
let  ${\mathcal D}'_{\rm S} = (\mathfrak{d},P',G)$ denote the lazy RWoGG, 
where $\mathfrak{d}(n)=1$ for any $n$. 

 It is trivial that $r$ is always recurrent either by ${\mathcal D}_{\rm S}$ or by ${\mathcal D}'_{\rm S}$. 
 Clearly, the degree of $G_n$ is unbounded asymptotic to $n$. 
 Interestingly, the recurrence/transience of $v_i$ depends on the growing speed.  

\begin{proposition}[Busy walk]\label{prop:star-busy}
 If $\sum_{n=1}^{\infty} \frac{1}{M(n)} = \infty$
   then $v_i$ for $1 \leq i < \infty$ is recurrent by ${\mathcal D}_{\rm S}$,  
  otherwise transient.  
\end{proposition}

\begin{proposition}[Lazy walk]\label{prop:star-lazy}
 If $\sum_{n=1}^{\infty} \frac{1}{M(n)} = \infty$
   then $v_i$ for $1 \leq i < \infty$ is recurrent by ${\mathcal D}'_{\rm S}$, 
  otherwise transient.  
\end{proposition}

\section{Concluding Remark}
 In this paper, we have developed a coupling method 
  to prove the recurrence and transience of a RWoGG, 
  by introducing the notion of LHaGG. 
 Then, we showed the phase transition 
   between the recurrence and transience 
   of random walks on a growing $k$-ary tree (Theorem~\ref{thm:karytree}) and 
   on a growing hypercube (Theorem~\ref{thm:cube}). 
 We have also shown other examples of LHaGG, 
  such as growing integer grid (Theorems~\ref{thm:box} and \ref{thm:box2}), and 
  growing level tree (Section~\ref{sec:tree}). 
 
%%%%  
 On the other hand, our examples are relatively specific. 
 It is a future work to develop an extended technique 
  to prove phase transitions for more general growing trees and integer grids. 
% Extension to $\{0,\ldots,N\}^n$ with an increasing $n$ is a future work. 

\section*{Acknowledgement}
This work is partly supported by 
  JST SPRING Grant Number JPMJSP2136 and 
  JSPS KAKENHI Grant Numbers JP21H03396, JP20K20884, JP22H05105 and JP23H01077.


\begin{thebibliography}{99}
\bibitem{Amir20}
 G.\ Amir, I.\ Benjamini, O.\ Gurel-Gurevich and G.\ Kozma, 
 Random walk in changing environment, 
 Stochastic Processes and their Applications, 130 (2020), 7463--7482. 
 
\bibitem{AR01}
 V.\ S.\ Anil Kumar and H.\ Ramesh, 
 Coupling vs. conductance for the Jerrum-Sinclair chain, 
 Random Struct. Algorithms, 18:1 (2001),  1--17. 

\bibitem{APR16}
 J.\ Augustine, G.\ Pandurangan and P.\ Robinson, 
 Distributed algorithmic foundations of dynamic networks, 
 SIGACT News, 47:1 (2016), 69--98. 

\bibitem{AKL08}
 C.\ Avin, M.\ Kousk\'y and Z.\ Lotler,  
 How to explore a fast-changing world (cover time of a simple random walk on evolving graphs), 
 In Proc. ICALP 2008, 121--132. 

\bibitem{AKL18}
 C.\ Avin, M.\ Kousk\'y, and Z.\ Lotler, 
 Cover time and mixing time of random walks on dynamic graphs, 
 Random Structures \& Algorithms, 52:4 (2018), 576--596. 

\bibitem{BW03}
  I.\ Benjamini and D.\ B.\ Wilson, 
  Excited random walk, 
  Electron. Comm. Probab., 8:9 (2003), 86–92.

\bibitem{CSZ20} 
 L.\ Cai, T.\ Sauerwald and L.\ Zanetti, 
 Random walks on randomly evolving graphs. 
 In Proc. SIROCCO 2020, 111--128. 

\bibitem{Cooper11} 
 C.\ Cooper, 
 Random walks, interacting particles, dynamic networks: randomness can be helpful, 
 In Proc. SIROCCO 2011, 1--14. 

\bibitem{CF03} 
 C.\ Cooper and A.\ Frieze, 
 Crawling on simple models of web graphs, 
 Internet Mathematics, 1:1 (2003), 57--90. 

\bibitem{Cormen} 
 T.\ H.\ Cormen, C.\ E.\ Leiserson, R.\ L.\ Rivest and C.\ Stein, 
 Introduction to Algorithms, fourth edition, 
 The MIT Press, 2022.


\bibitem{Davis90}
 B.\ Davis, 
 Reinforced random walk, 
 Probab. Theory Related Fields, 84:2 (1990), 203--229. 

\bibitem{DHMP17}
 A.\ Dembo, R.\ Huang, B.\ Morris and Y.\ Peres, 
 Transience in growing subgraphs via evolving sets, 
 Annales de l'Institute Henri Poincar\'e Probabilit\'es Statistques, 53:3 (2017), 1164--1180. 
% Ann. Inst. H. Poincaré Probab. Statist. 53 (3) 1164 - 1180, August 2017. https://doi.org/10.1214/16-AIHP751

\bibitem{DHS14C}
 A.\ Dembo, R.\ Huang and V.\ Sidoravicius, 
 Monotone interaction of walk and graph: recurrence versus transience, 
 Electronic Communications in Probability, 19:76 (2014), 1--12. 
% Electron. Commun. Probab. 19 1 - 12, 2014. https://doi.org/10.1214/ECP.v19-3607

\bibitem{DHS14J}
 A.\ Dembo, R.\ Huang and V.\ Sidoravicius, 
 Walking within growing domains: recurrence versus transience, 
 Electronic Journal of Probability, 19:106 (2014), 1--20.  
% Electron. J. Probab., 19 (2014), 1--20.  

\bibitem{DHZ19}
 A.\ Dembo, R.\ Huang and T.\ Zheng, 
 Random walks among time increasing conductances: heat kernel estimates,
  Probab. Theory Relat. Fields, 175 (2019), 397--445. 
%  Probab. Theory Relat. Fields 175, 397–445 (2019). https://doi.org/10.1007/s00440-018-0894-1

\bibitem{DR14}
 O.\ Denysyuk and L.\ Rodrigues, 
 Random walks on evolving graphs with recurring topologies, 
 In Proc. DISC 2014, 333--345.

\bibitem{DKL08}
 D.\ Dolgopyat, G.\ Keller and C.\ Liverani, 
 Random walk in Markovian environment, 
 Ann. Probab., 36:5 (2008), 1676--1710. 

\bibitem{Durrett} 
 R.\ Durrett, 
 Probability: Theory and Examples, 5th ed., 
 Cambridge University Press, 2019. 

\bibitem{FIORR21}
  D.\ Figueiredo, G.\ Iacobelli, R.\ Oliveira, B.\ Reed and R.\ Ribeiro, 
  On a random walk that grows its own tree, 
  Electron. J. Probab., 26 (2021), 1--40. 
%  , 2021. https://doi.org/10.1214/20-EJP574

\bibitem{Guruswami}
  V. Guruswami, 
  Rapidly mixing Markov chains: a comparison of techniques, 
  Unpublished, 2000. 

\bibitem{Huang19}
 R.\ Huang, 
 On random walk on growing graphs, 
 Annales de l'Institut Henri Poincar\'e - Probabilit\'es et Statistiques, 55:2 (2019), 1149--1162. 

\bibitem{IFN19}
G.\ Iacobelli, D.\ R.\ Figueiredo and G.\ Neglia,
 Transient and slim versus recurrent and fat: random walks and the trees they grow,
 Journal of Applied Probability, 56:3 (2019), 769--786.


\bibitem{KSS21}
 S.\ Kijima, N.\ Shimizu and T.\ Shiraga, 
 How many vertices does a random walk miss in a network with moderately increasing the number of vertices?, 
 in Proc. SODA 2021, 106--122.

\bibitem{KZ13}
 E.\ Kosygina and M.\ P.\ W.\ Zerner, 
 Excited random walks: results, methods, open problems, 
 Bulletin of the Institute of Mathematics Academia Sinica (New Series), 
% Bull. Inst. Math. Acad. Sin. (N.S.),
 8:1 (2013), 105--157. 

%\bibitem{Kozma06}
% G.\ Kozma, 
% Centrally excited random walk is recurrent, 
% Unpublished manuscript (2006).

%\bibitem{Kozma12}
%  G. Kozma, 
%  Reinforced random walk, 
%  Proc. of Europ. Cong. Math. (2012), 429--443.

\bibitem{KO11}
 F.\ Kuhn and R.\ Oshman, 
 Dynamic networks: models and algorithms, 
 SIGACT News, 42:1 (2011), 82--96. 

\bibitem{LMS18}
 I.\ Lamprou, R.\ Martin and P.\ Spirakis, 
 Cover time in edge-uniform stochastically-evolving graphs, 
 Algorithms, 11:10 (2018), 149. 

\bibitem{LevinPeres}
 D.\ A.\ Levin and Y.\ Peres, 
 Markov Chains and Mixing Times, 
 Am. Math. Soc., 2017. 

\bibitem{Lyons90}
 R.\ Lyons,
 Random walks and percolation on trees,
 The Annals of Probability, 18:3 (1990), 931--958. 

\bibitem{LyonsPeres}
 R.\ Lyons and Y.\ Peres, 
 Probability on tree and networks, 
 Cambridge University Press,  2017. 

\bibitem{MS18}
 O.\ Michail and P.\ G.\ Spirakis, 
 Elements of the theory of dynamic networks, 
 Communications of the ACM, 61:2 (2018), 72--81. 

\bibitem{SZ09}
 L.\ Saloff-Coste and J.\ Z\'u\~niga, 
 Merging for inhomogeneous finite Markov chains, part I: Singular values and stability, 
 Electron. J. Probab., 14 (2009), 1456--1494. 
 
\bibitem{SZ11}
 L.\ Saloff-Coste and J.\ Z\'u\~niga, 
 Merging for inhomogeneous finite Markov chains, part II: Nash and log-Sobolev inequalities, 
 Ann. Probab., 39 (2011), 1161--1203. 

\end{thebibliography}
\end{document}